\documentclass[reqno,11pt]{amsart}

\usepackage[margin=1in]{geometry}
\usepackage{amsmath,amssymb,amsthm,mathtools}
\usepackage{microtype}
\usepackage[hidelinks]{hyperref}

\allowdisplaybreaks
\numberwithin{equation}{section}

\newtheorem{theorem}{Theorem}[section]
\newtheorem{lemma}[theorem]{Lemma}
\newtheorem{fact}[theorem]{Fact}
\newtheorem{proposition}[theorem]{Proposition}

\newtheorem{claim}[theorem]{Claim}
\theoremstyle{definition}
\newtheorem{definition}[theorem]{Definition}

\newcommand{\cB}{\mathcal B}
\newcommand{\cH}{\mathcal H}
\newcommand{\EE}{\mathbb E}
\newcommand{\PP}{\mathbb P}
\newcommand{\HH}{\mathrm H}
\newcommand{\one}{\mathbf 1}
\newcommand{\poly}{\mathrm{poly}}
\newcommand{\llpoly}{\mathrel{\overset{\poly}{\ll}}}
\newcommand{\midtilde}[1]{\mathord{\mkern1mu\widetilde{\mkern-1mu#1\mkern-1mu}\mkern1mu}}
\newcommand{\e}{\mathrm e}
\def\eps{\varepsilon}

\title{Counting Near-Spanning Matchings in Latin Squares and Steiner Triple Systems}
\thanks{Research partially supported by NSF grant DMS-2300346 and Simons travel support SFI-MPS-TSM-00025377.}
\author{Yantao Tang}
\address{Department of Mathematics and Statistics, Georgia State University, Atlanta, GA 30303}
\email{ytang26@gsu.edu}
\author{Yi Zhao}
\address{Department of Mathematics and Statistics, Georgia State University, Atlanta, GA 30303}
\email{yzhao6@gsu.edu}
\date{}

\begin{document}

\begin{abstract}
Montgomery proved that for sufficiently large $n$, every Latin square of order $n$ has a partial transversal with $n-1$ cells, and every Steiner triple system of order $n$ has a matching with $\lfloor n/3\rfloor-1$ edges, thus confirming the Ryser--Brualdi--Stein conjecture for even $n$ and the conjecture of Brouwer. We prove sharp enumerative refinements of these results: 
there is an absolute constant $c>0$ such that, for sufficiently large $n$,
\begin{itemize}
\item every Latin square of order $n$ has $ \left((1\pm n^{-c})\frac{n}{\e^2}\right)^n$
partial transversals with $n-1$ cells;
\item every Steiner triple system of order $n$ has $ \left((1\pm n^{-c})\frac{n}{2\e^2}\right)^{\lfloor n/3\rfloor}$
matchings with $\lfloor n/3\rfloor-1$ edges. 
\end{itemize}
The first estimate confirms predictions of Montgomery and Kelly.  We also provide sharp estimates for the numbers of partial transversals and matchings of every smaller size.
\end{abstract}

\maketitle

\section{Introduction}

\subsection{Latin squares}

A Latin square of order $n$ is an $n\times n$ array on $n$ symbols in which every symbol appears exactly once in each row and in each column.  A \emph{partial transversal} is a set of cells using no row, column, or symbol more than once. A partial transversal with $i$ cells is an \emph{$i$-transversal}; an $n$-transversal is a (full) transversal.  A Latin square has an orthogonal mate if and only if its cells can be partitioned into transversals.  We refer to the surveys of Wanless~\cite{WanlessSurvey} and Montgomery~\cite{MontgomerySurvey} for further background, and to Pokrovskiy~\cite{PokrovskiySurvey} for related rainbow problems.

The Ryser--Brualdi--Stein conjecture asserts that every Latin square of order $n$ has an $(n-1)$-transversal, and an $n$-transversal if $n$ is odd.  The second assertion cannot be extended to even order, since the addition table of a cyclic group of even order has no transversal.  Koksma~\cite{Koksma} first proved a universal lower bound of $2n/3$, and Drake~\cite{Drake} improved this to $3n/4$.  Brouwer, de Vries and Wieringa~\cite{BrouwerDeVriesWieringa} and Woolbright~\cite{Woolbright} independently obtained $n-O(\sqrt n)$.  Shor~\cite{Shor} with a correction by Hatami and Shor~\cite{HatamiShor} reduced the error to $O(\log^2 n)$, and Keevash, Pokrovskiy, Sudakov and Yepremyan~\cite{KPSY} obtained $O(\log n/\log\log n)$.  Montgomery~\cite{MontgomeryRBS} proved the $n-1$ assertion for all sufficiently large Latin squares.  The transversal assertion for odd order remains open.

A separate line of work concerns the number of transversals.  Building on earlier estimates of McKay, McLeod and Wanless~\cite{McKayMcLeodWanless}, Taranenko~\cite{Taranenko} proved that every Latin square of order $n$ has at most $\bigl((1+o(1))n/\e^2\bigr)^n$ transversals.  Glebov and Luria~\cite{GlebovLuria} gave a shorter entropy proof and determined the exponential order of the maximum.  Kwan~\cite{KwanSTS} showed that almost every Latin square has $\bigl((1-o(1))n/\e^2\bigr)^n$ transversals.  More precise typical results include the work of Gould and Kelly~\cite{GouldKelly}, Eberhard, Manners and Mrazovi\'c~\cite{EMM}, and Bowtell and Montgomery~\cite{BowtellMontgomery}.

In his final remarks \cite[Section 11]{MontgomeryRBS}, Montgomery said that ``\emph{Using this [the semi-random method] in combination with the methods introduced here should show that any Latin square of order $n$ has $\exp(\Theta( n \log n))$ transversals with $n-1$ elements}''. In a recent survey \cite[Section 6]{KellyNibble}, Kelly said that ``\emph{It is likely that
the methods of Montgomery \cite{MontgomeryRBS} imply that every order-$n$ Latin square has at least $((1-o(1))n/e^2)^n$ partial transversals of size $n-1$}''.

In this paper we confirm these two predictions.
For $0\le j\le n$, let $T_j(L)$ denote the number of $j$-transversals of a Latin square $L$. 
Since some Latin squares have no full transversal, below we give sharp lower and upper bounds on $T_{n-1}(L)$ for all Latin squares $L$ of order $n$.

\begin{theorem}\label{thm:latin-main}
There is an absolute constant $c>0$ such that, for every sufficiently large integer $n$ and every Latin square $L$ of order $n$,
\[
 \left((1-n^{-c})\frac{n}{\e^2}\right)^n
 \le T_{n-1}(L)
 \le
 \left((1+n^{-c})\frac{n}{\e^2}\right)^n.
\]
\end{theorem}

Thus the universal count at size $n-1$ has the same leading exponential term as the extremal and typical counts of full transversals \cite{KwanSTS,Taranenko} mentioned earlier.

\subsection{Steiner triple systems}

A Steiner triple system of order $n$ is a $3$-uniform hypergraph in which every pair of vertices lies in exactly one edge.  Such a system exists only when $n\equiv1,3\pmod 6$.  A matching covering every vertex is \emph{perfect} and called a \emph{parallel class} of the system, and a matching covering all but one vertex is an \emph{almost parallel class}.  Let $m=\lfloor n/3\rfloor$.  An $m$-edge matching is a parallel class when $n=3m$ and an almost parallel class when $n=3m+1$. 
For $n=3m+1$, Wilson (see \cite{RosaColbourn}) observed that the projective Steiner triple systems of order $2^r-1$, with $r$ odd, have no almost parallel class; Bryant and Horsley~\cite{BryantHorsleyAlmost} later constructed another infinite family of such systems.  For $n=3m$, Bryant and Horsley~\cite{BryantHorsleyParallel} constructed an infinite family of Steiner triple systems with no parallel class.

Brouwer~\cite{BrouwerSTS} proved that every Steiner triple system contains a matching of size $n/3-O(n^{2/3})$ and conjectured that the size can be improved to $\lceil(n-4)/3\rceil$.  Alon, Kim and Spencer~\cite{AKS} reduced the number of vertices left uncovered by a largest matching to $O(n^{1/2}\log^{3/2}n)$, and Keevash, Pokrovskiy, Sudakov and Yepremyan~\cite{KPSY} improved this to $O(\log n/\log\log n)$.  Montgomery~\cite{MontgomeryRBS} finally proved Brouwer's conjecture for sufficiently large $n$, that is, every sufficiently large Steiner triple system has a matching of size $m-1$. 

Kwan~\cite{KwanSTS} proved that every Steiner triple system has at most $\bigl((1+o(1))n/(2\e^2)\bigr)^{n/3}$ perfect matchings and, when $n\equiv 3 \pmod6$, almost every Steiner triple system has at least $\bigl((1-o(1))n/(2\e^2)\bigr)^{n/3}$ perfect matchings.  Ferber and Kwan~\cite{FerberKwan} subsequently proved that almost all Steiner triple systems of order $n$ divisible by 3 admit a decomposition of almost all their triples into disjoint perfect matchings.  

For a Steiner triple system $S$, let $N_j(S)$ denote the number of $j$-edge matchings. The results of \cite{BryantHorsleyAlmost,BryantHorsleyParallel} show that $N_m(S)=0$ for some $S$. Below we give a sharp estimate on $N_{m-1}(S)$, the number of the largest matchings guaranteed in every sufficiently large Steiner triple system.

\begin{theorem}\label{thm:sts-main}
There is an absolute constant $c>0$ such that the following holds whenever $n\equiv1,3\pmod6$ is sufficiently large.  Let $S$ be a Steiner triple system of order $n$, and put $m=\lfloor n/3\rfloor$.  Then
\[
 \left((1-n^{-c})\frac{n}{2\e^2}\right)^m
 \le N_{m-1}(S)
 \le
 \left((1+n^{-c})\frac{n}{2\e^2}\right)^m.
\]
\end{theorem}


\subsection{Counting matchings of other sizes}

Theorems~\ref{thm:latin-main} and~\ref{thm:sts-main} concern the largest sizes that can be guaranteed in every sufficiently large Latin square and Steiner triple system.  We extend them to estimate the numbers of partial transversals and matchings of smaller sizes. Given integers $n\ge r\ge 0$, define 
\[
 A_{n,r}:=\binom nr^{\!3}\frac{(r!)^2}{n^r}
 \quad(r\le n),
 \qquad
 B_{n,r}:=\frac{n!}{(n-3r)!\,6^r r!\,(n-2)^r}
 \quad(r\le m),
\]
where $n\equiv1,3\pmod 6$ and $m=\lfloor n/3\rfloor$ in the definition of $B_{n, r}$.
We observe that $A_{n,r}$ and $B_{n,r}$ are expected numbers of matchings in related random hypergraphs. Indeed, in the complete $3$-partite $3$-uniform hypergraph with three classes of size $n$, an $r$-edge matching first chooses $r$ vertices from each class and then matches the first chosen class bijectively to each of the other two.  There are therefore $\binom nr^3(r!)^2$ such matchings.  If each triple is retained independently with probability $1/n$, a fixed matching survives with probability $n^{-r}$.  By linearity of expectation, $A_{n,r}$ is the expected number of $r$-edge matchings in this random hypergraph.
Similarly, the complete $3$-uniform hypergraph on $n$ vertices has
\[
 \binom n{3r}\frac{(3r)!}{(3!)^r r!}
 =\frac{n!}{(n-3r)!\,6^r r!}
\]
matchings of size $r$: choose their $3r$ vertices and partition them into $r$ unordered triples.  For $n\ge3$, retaining each triple independently with probability $1/(n-2)$ gives expected matching count $B_{n,r}$.  The two random hypergraphs have expected vertex degrees $n$ and $(n-1)/2$, respectively, agreeing with the degrees of the Latin-square hypergraph and a Steiner triple system.  

We use $X=Y\exp(\pm z)$ to mean $Ye^{-z}\le X\le Ye^z$.  The following theorem compares matching counts in an arbitrary Latin square and in a Steiner triple system with $A_{n,r}$ and $B_{n, r}$.

\begin{theorem}\label{thm:matching-profiles}
In the following statements, $L$ and $S$ are arbitrary Latin squares and Steiner triple systems of order $n$,
respectively, and $m=\lfloor n/3\rfloor$.

There exists an absolute constant \(C>0\) such that, for every $n\ge3$, 
\[
T_r(L)=A_{n,r}\exp{\left(\pm C\frac{r^3}{n^2}\right)}
\quad (0\le r\le n/3), \quad \text{and} \quad N_r(S)=B_{n,r}\exp{\left(\pm C\frac{r^3}{n^2}\right)}
\quad (0\le r\le m/3).
\]

There is an absolute constant \(c>0\) such that, for every \(\eta\in(0,1)\), there exist \(C'>0\) and \(n_0\in\mathbb N\) for which the following holds for every \(n\ge n_0\):
\[
T_r(L)=A_{n,r}\exp{\left(\pm C' n^{1-c}\right)}
\quad (\eta n\le r<n), \quad \text{and} \quad N_r(S)=B_{n,r}\exp{\left(\pm C' n^{1-c}\right)}\quad (\eta m\le r<m).
\]
\end{theorem}

The small-matching estimates show that $T_r(L)=(1+o(1))A_{n,r}$ and $N_r(S)=(1+o(1))B_{n,r}$ when $r=o(n^{2/3})$. The large-matching estimates determine the counts up to a multiplicative factor $\exp(o(n))$
when $r$ is a positive proportion of $n$.

\subsection*{Overview of the proofs}

The lower bounds of Theorems~\ref{thm:latin-main} and \ref{thm:sts-main} are proved in the same colored-graph setting.  A Latin square gives a $3$-partite $3$-uniform hypergraph whose vertex classes are the rows, columns, and symbols.  For a Steiner triple system, we choose a balanced tripartition and retain only triples that meet each part once.  In either case the resulting hypergraph can also be represented by a properly edge-colored bipartite graph, with hypergraph matchings corresponding to rainbow matchings.

Montgomery's proof for Latin squares starts from the corresponding properly colored $K_{n,n}$ and exposes random vertex and color sets.  A semi-random argument provides an almost-perfect rainbow matching in any subgraph that contains suitably sized random sets of vertices and colors, leaving a small balanced remainder; see Lemma~\ref{lem:mont-finish} (almost-cover lemma).  
Because algebraic obstructions of Latin squares prevent an arbitrary remainder from being absorbed, Montgomery chooses an identity color and builds two complementary structures: an addition structure transforms the remainder into a matching in that color, and an absorption structure ``absorbs'' the endpoints in this matching.  
We present this as Theorem~\ref{thm:reservoir} (completion theorem) and include its proof along with addition and absorption lemmas in Appendix~\ref{app:reservoir-proof}.

To prove our lower bounds, we combine the almost-perfect matching
count of Glock, Joos, Kim, K\"uhn and Lichev~\cite{GJKL} with Montgomery's approach.
We first fix both the random vertex and color
sets required by the almost-cover lemma and the completion theorem.  After deleting these sets, we count almost-perfect matchings in the remaining near-regular hypergraph. Each counted matching is then slightly trimmed so that the almost-cover lemma and the completion theorem can complete it to a matching with \(M-1\) edges, where \(M\) is the common size of the three vertex classes.
Finally, we bound how many counted matchings can lead to the same completed matching, which
preserves the required exponential lower bound. This procedure is summarized in Theorem~\ref{thm:lowerbound}, which can be applied directly to Latin squares.

For Steiner triple systems, we first randomly partition their vertices into three sets to obtain properly pseudorandom colored bipartite graphs, and then apply Theorem~\ref{thm:lowerbound}. A final double counting over the ordered tripartitions then gives the required lower bound.

For the upper bounds of Theorems~\ref{thm:latin-main}, \ref{thm:sts-main}, and \ref{thm:matching-profiles}, we prove an upper bound for the number of matchings of any size in all $k$-uniform hypergraphs whose maximum 2-degree is much smaller than its maximum vertex degree (Theorem~\ref{thm:entropy}). To prove it, we first follow Luria's entropy argument~\cite{Luria} to bound the number of perfect matchings in \emph{multi-hypergraphs} (hypergraphs with parallel edges) with the same degree condition. Then we extend the bound to matchings of arbitrary size in (simple) hypergraphs.

To prove Theorem~\ref{thm:matching-profiles}, we treat small and large matchings separately.  For small matchings, we count the ways to add one edge and compare the resulting recurrence with the ratios of consecutive terms of $A_{n,r}$ and $B_{n,r}$.  For large matchings, the upper bounds follow from Theorem~\ref{thm:entropy} directly. For the lower bound, we choose random vertex subsets and tune their sizes so that the theorem of \cite{GJKL} 
counts matchings of exactly the prescribed size in each suitable induced hypergraph.  A double counting over the subsets then gives the required lower bound.  This handles all sizes up to a sublinear distance from the maximum.  The remaining sizes follow by counting submatchings of the matchings provided by Theorems~\ref{thm:latin-main} and~\ref{thm:sts-main}.

\subsection*{Notation}
We use $\log$ to denote the natural logarithm with base $e$.
For a graph or hypergraph $G$, we denote by $V(G), E(G)$ its vertex set and edge set.
For a hypergraph $\cH$ and a vertex $v$ in its vertex set $V(\cH)$, we write $d_{\cH}(v)$, $\delta(\cH)$, $\Delta(\cH)$ and $\Delta_2(\cH)$ for the degree of $v$, minimum degree, maximum degree, and maximum 2-degree (maximum number of edges that contain two fixed vertices).

We use hierarchies of constants to specify their dependencies. We write $x\ll y$ to say that there is a positive increasing function $f:(0,1]\to \mathbb{R}$ such that the assertion under discussion is valid whenever $x\le f(y)$.  The stronger relation $x\llpoly y$ means that one may take $f(y)=y^K/K$ for some constant $K>0$. The functions behind a hierarchy with several constants are chosen from right to left.  We omit floors and ceilings unless they are crucial.  

\subsection*{Organization of the paper}

Section~\ref{sec:inputs} contains the colored-graph model, the almost-cover lemma and completion theorem, the matching-counting theorem of \cite{GJKL}, and the basic probabilistic and entropy facts.  Section~\ref{sec:lower} proves a general lower-bound theorem.  The entropy upper bound for arbitrary matching sizes is established in Section~\ref{sec:entropy}, and Theorems~\ref{thm:latin-main} and \ref{thm:sts-main} are proved in Sections~\ref{sec:latin} and~\ref{sec:sts}.  Section~\ref{sec:all-sizes} proves Theorem~\ref{thm:matching-profiles}, and Section~\ref{sec:conclusion} discusses related counting questions.  Appendix~\ref{app:montgomery-definitions} gives Montgomery's definitions of typicality and proper pseudorandomness.  Appendix~\ref{app:reservoir-proof} records the absorption and addition inputs and proves Theorem~\ref{thm:reservoir}.  Appendix~\ref{app:profile-proofs} contains the proofs of Claims~\ref{prof:sts-extension}, \ref{prof:restrictions}, and~\ref{prof:restriction-tuning}.
\section{Preliminary tools}\label{sec:inputs}

\subsection{Edge-colored bipartite graphs and linear 3-partite 3-uniform hypergraphs}


Suppose that $G$ is a properly edge-colored bipartite graph with vertex classes $A$ and $B$. Let $C(G)$ be the set of colors in $G$.  For $c\in C(G)$, let $E_c(G)$ be the set of edges in $G$ of color $c$.  For an edge subset $F\subseteq E(G)$, we denote by $V(F)$ and $ C(F)$ the sets of vertices and colors respectively of edges in $F$.  For a color subset $C\subseteq C(G)$, a matching of $G$ is \emph{$C$-rainbow} if its edges have distinct colors belonging to $C$.

Associate with \(G\) the \(3\)-partite \(3\)-uniform hypergraph \(\mathcal H(G)\) with vertex classes \(A\), \(B\), and \(C(G)\), and edge set $E(\cH(G)):=\bigl\{\{a,b,c\}:ab\in E(G)\text{ has color }c\bigr\}$. 
Matchings in $\cH(G)$ correspond exactly to rainbow matchings in $G$.
The hypergraph $\cH(G)$ is \emph{linear}, which means that every pair of vertices lies in at most one edge.  Indeed, every edge of \(G\) has a unique color, and properness ensures that a vertex and a color occur together in at most one edge.

Montgomery \cite{MontgomeryRBS} worked on a larger class of colored graphs than complete bipartite graphs, \emph{$(M,p,\varepsilon)$-properly pseudorandom} bipartite graphs. Its definition \cite[Definition~3.11]{MontgomeryRBS} is long and technical, so we present it in Appendix~\ref{app:montgomery-definitions}. The following proposition from \cite{MontgomeryRBS} shows that every properly edge-colored $K_{n,n}$ with $n$ colors is $(n,1,\varepsilon)$-properly pseudorandom for any $\varepsilon>0$. 

\begin{proposition}[{\cite[Proposition~3.12]{MontgomeryRBS}}]\label{prop:mont-pseudorandom}
Let $1/n\llpoly\eta\llpoly\varepsilon\le1$.  Let $G$ be a properly colored bipartite graph with vertex classes $A$ and $B$ such that $|A|=|B|=n$, $|C(G)|\ge n$ and $\delta(G)\ge(1-\eta)n$.  Suppose that each color of $G$ appears at least $(1-\eta)n$ times on the edges of $G$.
Then, $G$ is $(n,1,\varepsilon)$-properly pseudorandom.
\end{proposition}


The next lemma is the semi-random almost-cover lemma of \cite{MontgomeryRBS}. It says that every choice of $qM$ vertices in each of the three classes that contains large random subsets admits a matching leaving at most $\eta M$ vertices uncovered in each class.  
It only requires the host hypergraph to be \emph{$(M,p,\varepsilon)$-typical}, which is weaker than $(M,p,\varepsilon)$-properly pseudorandom; see Appendix~\ref{app:montgomery-definitions} for their definitions.

\begin{lemma}[{\cite[Theorem~4.1]{MontgomeryRBS}}]\label{lem:mont-finish}
Let $M^{-1}\llpoly\varepsilon\llpoly\eta\llpoly p,q\le1$.  Let $2q/3\le q_A,q_B,q_C\le q$.  Let $\cH$ be an $(M,p,\varepsilon)$-typical linear $3$-partite $3$-uniform hypergraph with vertex classes $A$, $B$ and $C$. 
Choose $A'\subseteq A$, $B'\subseteq B$ and $C'\subseteq C$ independently by including each element with probability $q_A$, $q_B$ and $q_C$, respectively. 
Then, with high probability the following holds.

Given any sets $\widehat A\subseteq A$, $\widehat B\subseteq B$ and $\widehat C\subseteq C$ of size $qM$ such that $A'\subseteq\widehat A$, $B'\subseteq\widehat B$ and $C'\subseteq\widehat C$, there is a matching in $\cH[\widehat A\cup\widehat B\cup\widehat C]$ with at least $qM-\eta M$ edges.
\end{lemma}

The following completion theorem is extracted from Montgomery's proof of
\cite[Theorem~3.1, pp.~19--20]{MontgomeryRBS}.  
Its proof and two component lemmas are given in
Appendix~\ref{app:reservoir-proof}.

\begin{theorem}\label{thm:reservoir}
Let $M^{-1}\ll p\le1$ and $M^{-1}\llpoly\varepsilon\llpoly\eta\llpoly\beta\llpoly\log^{-1}M$.
Let $G$ be an $(M,p,\varepsilon)$-properly pseudorandom bipartite graph with classes $A,B$ and exactly $M$ colors.  Choose $V_0\subseteq A\cup B$ and $C_0\subseteq C(G)$ by including each element independently with probability $1/4$.  With probability $1-o(1)$, there are sets
\[
 U_A\subseteq A\cap V_0,\qquad U_B\subseteq B\cap V_0,\qquad U_C\subseteq C_0,
 \qquad |U_A|=|U_B|=|U_C|=\beta M+1,
\]
such that every matching $P$ in $\cH(G)\bigl[(A\setminus U_A)\cup(B\setminus U_B)\cup( C(G)\setminus U_C)\bigr]$ that leaves at most $\eta M$ elements uncovered in each of \(A\setminus U_A\), \(B\setminus U_B\), and \(C(G)\setminus U_C\) can be extended to a matching with $M-1$ edges in $\cH(G)$.
\end{theorem}

\subsection{Counting almost-perfect matchings}
Glock, Joos, Kim, K\"uhn and Lichev~\cite{GJKL} proved a lower bound on the number of matchings that avoid certain {forbidden submatchings} $\mathcal{C}$ (called \emph{conflicts}) in near-regular $k$-uniform hypergraphs with small 2-degree. We only use the special case $\mathcal{C}=\emptyset$.

\begin{lemma}[{\cite[Theorem~3.5]{GJKL}}]\label{thm:gjkl}
For every fixed integer $k\ge2$, there is $\zeta_0>0$ such that the following holds.  For every $0<\zeta<\zeta_0$, there is $d_0=d_0(k,\zeta)$ such that every $k$-uniform hypergraph $\cH$ on $v$ vertices satisfying $d\ge d_0$ and
\[
 v\le\exp(d^{\zeta^3}),
 \qquad
 (1-d^{-\zeta})d\le\delta(\cH)\le\Delta(\cH)\le d,
 \qquad
 \Delta_2(\cH)\le d^{1-\zeta}
\]
has at least
\[
 \left(\frac{(1-d^{-\zeta^4})d}{\e^{k-1}}\right)^s
\]
matchings of size $s:=\lfloor(1-d^{-\zeta^3})v/k\rfloor$.
\end{lemma}

\subsection{Concentration inequalities and entropy facts}

We use the following standard form of Chernoff's inequality.

\begin{lemma}
Let $Y=\sum_{i=1}^r Y_i$ and $\mu=\EE Y >0$, where the $Y_i$ are independent Bernoulli random variables.  For every $0\le t\le\mu$,
\[
 \PP(|Y-\mu|\ge t)\le 2\exp\left(-\frac{t^2}{3\mu}\right).
\]
\end{lemma}

We also use the following bounded-differences inequality for a random subset,
which follows from \cite[Theorem~(7.4) and Example~(7.3)]{McDiarmid1989}.
We state it with a weakened constant.

\begin{lemma}[Bounded differences for a random subset]
\label{lem:subset-concentration}
Let $U$ be a finite set, let $1\le u\le |U|$ be an integer, and let
$X$ be a uniformly random $u$-subset of $U$.
Suppose that $f:\binom{U}{u}\to\mathbb R$ satisfies
$|f(A)-f(B)|\le c$ whenever $|A\triangle B|=2$, where $c>0$.
Then, for every $z>0$,
\begin{equation}\label{prof:subset-tail}
 \PP\bigl(|f(X)-\EE f(X)|\ge z\bigr)
 \le 2\exp\left(-\frac{z^2}{8uc^2}\right).
\end{equation}
\end{lemma}

We use the following elementary properties of information entropy. The \emph{entropy} of a discrete random variable $Y$ is defined as $\HH(Y):=-\sum_y p_y\log p_y$, where $p_y= \PP(Y=y)$. 
If $Z$ is another discrete random variable, define the conditional entropy 
$\HH(Y\mid Z):=\sum_z\PP(Z=z)\HH(Y\mid Z=z)$.

\begin{fact}\label{lem:entropy-facts}
Let $Y,Y_1,\ldots,Y_r,Z$ be discrete random variables. Then the following hold.
\begin{enumerate}
\item[\textnormal{(i)}] If $Y$ is uniform on a set $\Omega$, then $\HH(Y)=\log|\Omega|$.
\item[\textnormal{(ii)}] The chain rule gives
\[
 \HH(Y_1,\ldots,Y_r)=\sum_{i=1}^r\HH(Y_i\mid Y_1,\ldots,Y_{i-1}).
\]
\item[\textnormal{(iii)}] Suppose that, whenever $\PP(Z=z)>0$, there are at most $N(z)$ possible values of $Y$ after conditioning on $Z=z$.  Then
\[
 \HH(Y\mid Z)\le \EE\log N(Z).
\]
\item[\textnormal{(iv)}] If $U$ is a positive random variable, then for every $z$ with $\PP(Z=z)>0$,
\[
 \EE[\log U\mid Z=z]\le \log\EE[U\mid Z=z],
\]
whenever the expectations are finite.
\end{enumerate}
\end{fact}

Part~\textnormal{(iii)} follows from the fact that, among distributions on at most $N$ possible values, the entropy is at most $\log N$.  Part~\textnormal{(iv)} is Jensen's inequality for the concave function $\log x$.

\section{The general lower bound}\label{sec:lower}

The following theorem gives a lower bound on the number of near-spanning rainbow matchings in a properly pseudorandom bipartite graph, and it will be applied to Latin squares and Steiner triple systems. 

\begin{theorem}\label{thm:lowerbound}
Fix $0<p_0\le1$ and a sufficiently small $a>0$.  There are constants $\sigma,C>0$ such that the following holds for all sufficiently large $M$.  Let $p\in[p_0,1]$, and let $G$ be an $(M,p,M^{-a})$-properly pseudorandom bipartite graph with classes $A,B$ and exactly $M$ colors.  Then $G$ has at least
\[
 \left(\frac{pM}{\e^2}\right)^M\exp\bigl(-CM^{1-\sigma}\bigr)
\]
rainbow matchings with $M-1$ edges.  
\end{theorem}

\begin{proof}
Let $\zeta_0$ be the constant in Lemma~~\ref{thm:gjkl} for $k=3$.  Choose sufficiently small
$0<\zeta<\min\{\zeta_0,a,1/3\}$ and then choose $\zeta^4<\delta<\zeta^3$.  Next choose $\eta,\beta$ as fixed negative powers of $M$ so that
\begin{equation}\label{eq:parameter-hierarchy}
 M^{-1}\llpoly M^{-a}
 \llpoly\eta
 \llpoly\beta
 \llpoly M^{-\zeta}
 \llpoly M^{-\delta}
 \llpoly\log^{-1}M.
\end{equation}
In particular, \eqref{eq:parameter-hierarchy} gives $\beta M=o(M^{1-\zeta})$.  Put $Q:=M^{1-\delta}$, $q:=Q/M$, and $\theta:=3q/4$.

Choose $V_0\subseteq A\cup B$ and $C_0\subseteq C(G)$ independently, with each element included with probability $1/4$.  Conditional on $V_0,C_0$, include each element of $A\setminus V_0$, $B\setminus V_0$, and $ C(G)\setminus C_0$ independently with probability $q$, obtaining $W_A,W_B,W_C$, respectively. 
Equivalently, the elements of $W_A,W_B,W_C$ were selected in two rounds, first with probability $3/4$, second with probability $q$.
Thus, $W_A,W_B,W_C$ are independent random subsets of $A,B,C(G)$, respectively, with each element included with probability $\theta=3q/4$.

We first prove the concentration estimates for the sizes of these sets.  Since $\theta M= \frac{3}{4}M^{1-\delta}$ and $\delta<1/3$, we have $M^{2/3}\le\theta M$ as $M$ is sufficiently large. Chernoff's inequality therefore gives
\[
 \PP\bigl(\bigl||W_A|-\theta M\bigr|\ge M^{2/3}\bigr)
 \le 2\exp\bigl(-\Omega(M^{1/3+\delta})\bigr).
\]
The same bound holds for $W_B,W_C$.  Applying Chernoff's inequality to $A\cap V_0$, $B\cap V_0$, and $C_0$ gives failure probability at most $2\exp(-\Omega(M^{1/3}))$ for each of the analogous estimates with mean $M/4$.  The union bound over these six events shows that the estimates
\begin{align}
 |W_A|,|W_B|,|W_C|&=\theta M+O(M^{2/3}),\label{eq:w-set-sizes}\\
 |A\cap V_0|,|B\cap V_0|,|C_0|&=M/4+O(M^{2/3})\label{eq:v0-c0-sizes}
\end{align}
hold simultaneously with probability $1-o(1)$.

Let $W:=W_A\cup W_B\cup W_C$.  Fix $x\in V(\cH(G))$ and condition on the event $x\notin W$.  Distinct edges of $\cH(G)$ containing $x$ use pairwise distinct vertices in each of the other two coordinate classes, because $\cH(G)$ is linear.  Hence the indicators that these edges survive the deletion of $W$ are independent Bernoulli variables, each with success probability $(1-\theta)^2$.  Since $G$ is $(M,p,M^{-a})$-properly pseudorandom, \textnormal{(F2)} in Definition~\ref{def:proper-pseudorandom} states $d_{\cH(G)}(x)=(1\pm M^{-a})pM$.  Thus, for every $x\notin W$, 
\[
 {\EE}\bigl[d_{\cH(G)-W}(x)\bigr]
 =(1-\theta)^2d_{\cH(G)}(x)
 =(1-\theta)^2(1\pm M^{-a})pM
\]
Since $\theta=O(M^{-\delta})$, $\delta<a$ and $p\in[p_0,1]$, it follows that 
${\EE}\bigl[d_{\cH(G)-W}(x) \,|\, x\notin W \bigr] =(1+O(M^{-\delta}))pM=\Theta(M)$.
 In particular, this expectation is at least $M^{2/3}$ as $M$ is sufficiently large. Chernoff's inequality, with deviation $M^{2/3}$, therefore gives
\[
 \PP\left(\left|d_{\cH(G)-W}(x)-(1-\theta)^2d_{\cH(G)}(x)\right|\ge M^{2/3}\,\middle|\,x\notin W\right)
 \le 2\exp\bigl(-\Omega(M^{1/3})\bigr).
\]
Summing these probabilities for at most $3M$ vertices shows that all $x\notin W$ simultaneously satisfy
\begin{equation}\label{eq:degree-after-W}
 d_{\cH(G)-W}(x)=(1-\theta)^2d_{\cH(G)}(x)\pm M^{2/3}
\end{equation}
with probability $1-o(1)$.

We now apply Lemma~\ref{lem:mont-finish} and Theorem~\ref{thm:reservoir}.  By \textnormal{(F2)} of proper pseudorandomness, $\cH(G)$ is $(M,p,M^{-a})$-typical.  Since $p\ge p_0$ and $q=M^{-\delta}$, \eqref{eq:parameter-hierarchy} gives $M^{-1}\llpoly M^{-a}\llpoly\eta\llpoly p,q$, while $2q/3\le\theta\le q$.  Hence Lemma~\ref{lem:mont-finish} applies to $\cH(G)$ with $A'=W_A$, $B'=W_B$, $C'=W_C$, and $q_A=q_B=q_C=\theta$.  The condition $M^{-1}\ll p$ follows from $p\ge p_0$, and the hierarchy gives all the polynomial conditions of Theorem~\ref{thm:reservoir}, so that theorem applies to $V_0,C_0$.

A union bound shows that \eqref{eq:w-set-sizes}, \eqref{eq:v0-c0-sizes}, \eqref{eq:degree-after-W}, and the following two properties hold simultaneously with probability $1-o(1)$.

\begin{enumerate}
\item[\textnormal{(P1)}] For every $\widehat A\subseteq A$, $\widehat B\subseteq B$, and $\widehat C\subseteq C(G)$ satisfying $|\widehat A|=|\widehat B|=|\widehat C|=Q$, $W_A\subseteq\widehat A$, $W_B\subseteq\widehat B$, and $W_C\subseteq\widehat C$, the hypergraph $\cH(G)[\widehat A\cup\widehat B\cup\widehat C]$ contains a matching with at least $Q-\eta M$ edges.

\item[\textnormal{(P2)}] There are sets $U_A\subseteq A\cap V_0$, $U_B\subseteq B\cap V_0$, and $U_C\subseteq C_0$ with $|U_A|=|U_B|=|U_C|=\beta M+1$ such that every matching $P$ in $\cH(G)[(A\setminus U_A)\cup(B\setminus U_B)\cup(C(G)\setminus U_C)]$ that leaves at most $\eta M$ elements uncovered in each of the three classes can be extended to a matching with $M-1$ edges in $\cH(G)$.
\end{enumerate}

Fix a realization for which \eqref{eq:w-set-sizes}--\eqref{eq:degree-after-W}, and properties \textnormal{(P1)} and \textnormal{(P2)} hold, and henceforth regard $V_0,C_0,W_A,W_B,W_C$ as deterministic.  Fix corresponding sets $U_A,U_B,U_C$ from \textnormal{(P2)}, and put $r_0:=|U_A|=|U_B|=|U_C|=\beta M+1$.

Put
\[
 s_0:=\max\{|W_A|,|W_B|,|W_C|\}.
\]
By \eqref{eq:w-set-sizes} and \eqref{eq:v0-c0-sizes}, for each $X\in\{A,B,C\}$ we may enlarge $W_X$ to a set $\midtilde{W}_X$ of size $s_0$ by adding $O(M^{2/3})$ elements outside $V_0$ when $X\in\{A,B\}$, and outside $C_0$ when $X=C$.  Define
\[
 A^*:=A\setminus(U_A\cup\midtilde{W}_A),\quad
 B^*:=B\setminus(U_B\cup\midtilde{W}_B),\quad
 C^*:= C(G)\setminus(U_C\cup\midtilde{W}_C).
\]
The three sets have common size $N:=M-r_0-s_0$.  Let $\cH^*:=\cH(G)[A^*\cup B^*\cup C^*]$.

The passage from $\cH(G)-W$ to $\cH^*$ deletes $r_0+O(M^{2/3})$ vertices from each coordinate class. 
Fix $x\in V(\cH^*)$.   By linearity, the degree of $x$ therefore decreases by at most $2r_0+O(M^{2/3})$.  Combining this with \eqref{eq:degree-after-W} and $d_{\cH(G)}(x)=(1\pm M^{-a})pM$, we obtain that
\begin{equation}\label{eq:residual-degrees}
 d_{\cH^*}(x)
 =(1-\theta)^2pM\pm O\bigl(M^{1-a}+r_0+M^{2/3}\bigr)
 =(1-\theta)^2pM\pm o(M^{1-\zeta}),
\end{equation}
where the last equality follows because $\zeta<a$, $\zeta<1/3$, and \eqref{eq:parameter-hierarchy} gives $r_0=o(M^{1-\zeta})$.  Set $d:=\Delta(\cH^*)$.  Since $p\ge p_0$, it follows that $d=\Theta(M)$ and $\Delta(\cH^*)-\delta(\cH^*)=o(M^{1-\zeta})=o(d^{1-\zeta})$.  Consequently, 
\[
 (1-d^{-\zeta})d\le\delta(\cH^*)\le\Delta(\cH^*)\le d.
\]
Since $M$ is sufficiently large, $d=\Theta(M)$ and $|V(\cH^*)|=3N=O(M)$, we have $d\ge d_0(3,\zeta)$ and $|V(\cH^*)|\le\exp(d^{\zeta^3})$.  Since $\cH^*$ is linear, $\Delta_2(\cH^*)\le1\le d^{1-\zeta}$.

Applying Lemma~~\ref{thm:gjkl} to $\cH^*$ with $k=3$, we obtain at least
\[
 Z:=\left(\frac{(1-d^{-\zeta^4})d}{\e^2}\right)^{m_0}
 \quad\text{matchings of size}\quad
 m_0:=\left\lfloor(1-d^{-\zeta^3})N\right\rfloor.
\]
Let $u:=N-m_0$.  Since $d=\Theta(M)$, $N=O(M)$ and $\delta<\zeta^3$, we have $u=d^{-\zeta^3}N+O(1)=O(M^{1-\zeta^3})=o(Q)$.  Equation~\eqref{eq:w-set-sizes} and the definition of $s_0$ give $s_0=3Q/4+O(M^{2/3})=3Q/4+o(Q)$, while \eqref{eq:parameter-hierarchy} and $\delta<\zeta$ give $r_0=o(Q)$.  Hence $m_0=N-u=(M-r_0-s_0)-u=M-O(Q)$.

Since $p\ge p_0$, $\theta=O(M^{-\delta})$, and $\delta<\zeta$, \eqref{eq:residual-degrees} gives $d=pM(1+O(M^{-\delta}))$.  As $d=\Theta(M)$ and $\zeta^4<\delta$, it follows that
\[
 \frac{(1-d^{-\zeta^4})d}{pM}=1+O(M^{-\zeta^4}).
\]
Using $\e^{-2|x|}\le1+x\le\e^{2|x|}$ for $|x|\le1/2$, together with $m_0=M-O(Q)$ and $pM/\e^2=\Theta(M)$, we obtain
\begin{equation}\label{eq:log-Z}
\begin{aligned}
 Z
 &=\left(\frac{pM}{\e^2}\right)^{m_0}
   \left(\frac{(1-d^{-\zeta^4})d}{pM}\right)^{m_0}\\
 &\ge\left(\frac{pM}{\e^2}\right)^M
   \exp\bigl(-O(M^{1-\zeta^4}+Q\log M)\bigr)\\
 &\ge\left(\frac{pM}{\e^2}\right)^M
   \exp\bigl(-O(M^{1-\zeta^4})\bigr).
\end{aligned}
\end{equation}
The last inequality uses $Q=O(M^{1-\delta})$ and $\delta>\zeta^4$, which imply $Q\log M=o(M^{1-\zeta^4})$.

We next delete some edges from each counted matching and extend the remaining matching to one with $M-1$ edges.  Put $h:=Q-s_0-u$.  Since $s_0=3Q/4+o(Q)$ and $u=o(Q)$, we have $h=(1/4+o(1))Q$, and hence $1\le h\le m_0$.  For each counted matching $M_0$, choose arbitrary $h$ of its edges and let $P_0$ be the remaining matching.  In each of $A\setminus U_A$, $B\setminus U_B$, and $ C(G)\setminus U_C$, the matching $P_0$ leaves exactly $(M-r_0)-|P_0|=s_0+u+h=Q=qM$ elements uncovered.  Denote these three uncovered sets by $X_A,X_B,X_C$, respectively.  Since $P_0\subseteq M_0\subseteq\cH^*$, no element of $\midtilde{W}_A\cup\midtilde{W}_B\cup\midtilde{W}_C$ is covered by $P_0$, so $X_A,X_B,X_C$ contain $W_A,W_B,W_C$, respectively.

For each counted matching $M_0$, choose a matching $M_{\mathrm f}$ with at least $Q-\eta M$ edges in $\cH(G)[X_A\cup X_B\cup X_C]$, as supplied by \textnormal{(P1)}.  Since $X_A,X_B,X_C$ are the sets left uncovered by $P_0$, the union $P_0\cup M_{\mathrm f}$ is a matching in $\cH(G)$.  It uses no element of $U_A\cup U_B\cup U_C$ and leaves exactly $Q-|M_{\mathrm f}|\le\eta M$ elements uncovered in each of $A\setminus U_A$, $B\setminus U_B$, and $C(G)\setminus U_C$.  Property \textnormal{(P2)} therefore gives an $(M-1)$-edge matching in $\cH(G)$ containing $P_0\cup M_{\mathrm f}$; choose one and denote it by $T(M_0)$.

To bound the number of distinct completed matchings, fix an $(M-1)$-edge matching $T$ in $\cH(G)$ and consider the counted matchings $M_0$ satisfying $T(M_0)=T$.  For each such $M_0$, we have $|P_0|=m_0-h=M-r_0-Q$, $P_0\subseteq T$, and $|T\setminus P_0|=r_0+Q-1$.  Hence there are at most $\binom{M-1}{r_0+Q-1}$ possibilities for $P_0$.  Once $P_0$ is fixed, $M_0$ is obtained by adjoining $h$ edges of $\cH^*$.  Since $|E(\cH^*)|\le |E(G)|\le M^2$, there are at most $\binom{M^2}{h}$ possibilities for $M_0$.  Hence the number of counted matchings $M_0$ satisfying $T(M_0)=T$ is at most
\[
 F_M:=\binom{M-1}{r_0+Q-1}\binom{M^2}{h}.
\]
Since $r_0=o(Q)$ and $h=O(Q)$,
\[
 F_M\le M^{r_0+Q+2h}
 =\exp(O(Q\log M))
 =\exp\bigl(o(M^{1-\zeta^4})\bigr).
\]
Thus $\cH(G)$ has at least $Z/F_M$ distinct matchings with $M-1$ edges.  By \eqref{eq:log-Z},
\[
 \frac{Z}{F_M}
 \ge\left(\frac{pM}{\e^2}\right)^M
 \exp\bigl(-O(M^{1-\zeta^4})\bigr).
\]
Taking $\sigma=\zeta^4$ and choosing $C=C(p_0,a)>0$ larger than the implicit constant gives the stated lower bound.  These matchings correspond exactly to rainbow matchings in $G$, proving the theorem.
\end{proof}
\section{An entropy upper bound}\label{sec:entropy}

Luria~\cite[Theorem~3.1]{Luria} proved that a $D$-regular $k$-uniform hypergraph on $kt$ vertices with maximum 2-degree $o(D)$ has at most $((1+o(1))D/\e^{k-1})^t$ perfect matchings. 
In this section we extend this result to count matchings of arbitrary size in any $k$-uniform hypergraph on $N$ vertices with $\Delta(\cH)\le D$ and $\Delta_2(\cH)=o(D)$.
For \(D\)-regular graphs, Ilinca and Kahn~\cite[Theorem~1.2]{IlincaKahn} give the corresponding estimate with a multiplicative error \(\exp(O(N\log D/D))\).
We thus restrict our result to $k\ge3$.

For a hypergraph $\cH$ and an integer $t\ge0$, let $M_t(\cH)$ denote the number of its $t$-edge matchings. 
\begin{theorem}\label{thm:entropy}
Fix an integer $k\ge3$.  There are constants $C_k,\rho_k>0$ such that the following holds.  Let $\cH$ be a $k$-uniform hypergraph on $N$ vertices satisfying $\Delta(\cH)\le D$ and $\Delta_2(\cH)\le\lambda$, where $D>0$ and $\lambda\ge0$.  Put
$\rho:=\left(1+\binom{k}{2}\lambda\right)/D$.
If $\rho\le\rho_k$, then, for every integer $1\le t\le\lfloor N/k\rfloor$, with $\alpha:=kt/N$,
\begin{equation}\label{eq:entropy-general-bound}
 M_t(\cH)
 \le
 \left(\frac{D}{\alpha\e^{k-1}}\right)^t
 (1-\alpha)^{-(N-kt)}
 \exp\left(\frac{C_kN}{k}\rho^{1/(k-1)}\right).
\end{equation}
When $\alpha=1$, the factor $(1-\alpha)^{-(N-kt)}$ is interpreted as $1$ (by the convention $0^0=1$).
\end{theorem}

\begin{proof}
\textbf{Perfect matchings.}
We first prove the case $N=kt$ in the stronger setting where $\cH$ may have parallel edges.  Parallel copies are distinguishable, and degrees and 2-degrees count multiplicity.  The desired bound in this case is
\begin{equation}\label{eq:entropy-perfect-bound}
 M_t(\cH)\le
 \left(\frac{D}{\e^{k-1}}\exp\bigl(C_k\rho^{1/(k-1)}\bigr)\right)^t.
\end{equation}
If $\cH$ has no perfect matching, the bound is immediate.  Otherwise let $X$ be a uniformly random perfect matching of $\cH$.  By Fact~\ref{lem:entropy-facts}\textnormal{(i)},
$\HH(X)=\log M_t(\cH)$.

For each $v\in V(\cH)$, let $X_v$ denote the edge of $X$ that contains $v$.  The vector $(X_v)_{v\in V(\cH)}$ determines $X$, and conversely $X$ determines this vector.  Hence $\HH((X_v)_{v\in V(\cH)})=\HH(X)$.

Independently of $X$, choose $\tau_v$ uniformly from $[0,1]$ for each $v\in V(\cH)$, with all these choices independent.  With probability one the priorities are distinct.  Fix such a priority vector $\tau$, order the vertices in decreasing order of priority, and let
$R_v(\tau):=(X_u:\tau_u>\tau_v)$ be the variables revealed before $X_v$.  Applying the chain rule from Fact~\ref{lem:entropy-facts}\textnormal{(ii)} in this deterministic order gives
\[
 \HH(X)=\sum_{v\in V(\cH)}\HH\bigl(X_v\mid R_v(\tau)\bigr).
\]
Since the left-hand side does not depend on $\tau$, averaging this identity over the random priorities yields
\begin{equation}\label{eq:entropy-chain}
 \HH(X)=\EE_\tau\sum_{v\in V(\cH)}\HH\bigl(X_v\mid R_v(\tau)\bigr).
\end{equation}

Fix $\tau$.  For each possible value $r$ of $R_v(\tau)$, let $N_v(r,\tau)$ be the number of edges $e$ containing $v$ for which there is a perfect matching that contains $e$ and agrees with $r$ on all previously revealed variables.  Thus, after conditioning on $R_v(\tau)=r$, the random variable $X_v$ has at most $N_v(r,\tau)$ possible values.  Fact~\ref{lem:entropy-facts}\textnormal{(iii)} therefore gives
\[
 \HH\bigl(X_v\mid R_v(\tau)\bigr)
 \le \EE_X\bigl[\log N_v(R_v(\tau),\tau)\bigr].
\]
Write $N_v:=N_v(R_v(\tau),\tau)$.  Let $A_v$ be the event that $v$ has the largest priority among the $k$ vertices of $X_v$.  If $A_v$ does not occur, then some $u\in X_v\setminus\{v\}$ has $\tau_u>\tau_v$, so $X_u=X_v$ is one of the variables revealed before $X_v$.  Once the earlier variables are fixed, $X_v$ is then already determined, and hence $N_v=1$.  Using $N_v=1$ outside $A_v$, the preceding entropy bound and \eqref{eq:entropy-chain} give
\begin{equation}\label{eq:entropy-first}
 \HH(X)\le \EE_{X,\tau}\sum_{v\in V(\cH)}\one_{A_v}\log N_v.
\end{equation}

We now estimate the contribution of a fixed vertex $v$.  First fix the matching $X$.  Since $\tau_v$ is uniform on $[0,1]$, we may condition on its value and integrate over $x$:
\[
 \EE_\tau[\one_{A_v}\log N_v\mid X]
 =\int_0^1
   \EE_\tau[\one_{A_v}\log N_v\mid X,\tau_v=x] \,dx.
\]
Fix $x\in(0,1)$.  Once $X$ and $\tau_v=x$ are fixed, the event $A_v$ occurs precisely when each of the other $k-1$ vertices of $X_v$ has priority less than $x$.  These $k-1$ priorities are independent and uniform on $[0,1]$, so $\PP(A_v\mid X,\tau_v=x)=x^{k-1}$.
It follows that
\[
\begin{aligned}
 \EE_\tau[\one_{A_v}\log N_v\mid X,\tau_v=x]
 &=x^{k-1}\EE_\tau[\log N_v\mid X,\tau_v=x,A_v]\\
 &\le x^{k-1}\log\EE_\tau[N_v\mid X,\tau_v=x,A_v],
\end{aligned}
\]
where the inequality is Jensen's inequality, in the form stated in Fact~\ref{lem:entropy-facts}\textnormal{(iv)}.  Integrating this estimate over $x$, then summing over $v$ and averaging over $X$ in \eqref{eq:entropy-first}, gives
\begin{equation}\label{eq:entropy-integral}
 \HH(X)\le
 \EE_X\sum_{v\in V(\cH)}\int_0^1 x^{k-1}
 \log\EE_\tau[N_v\mid X,\tau_v=x,A_v] \,dx.
\end{equation}

We next estimate the conditional expectation inside the logarithm.  Fix a realization of $X$.  Let $\cB_X$ be the set of edges $f\in E(\cH)\setminus X$ that meet some edge of $X$ in at least two vertices.  In particular, every $f\in E(\cH)\setminus X$ that has the same vertex set as an edge of $X$ belongs to $\cB_X$.  
The $t$ edges of $X$ contain $t\binom{k}{2}$ vertex pairs, and each such pair lies in at most $\lambda$ edges of $\cH$ because of $\Delta_2(\cH)\le\lambda$.  Hence
$|\cB_X|\le t\binom{k}{2}\lambda$.  If $b_v$ is the number of edges of $\cB_X$ containing $v$, then double counting incidences gives
\begin{equation}\label{eq:bad-incidences}
 \sum_{v\in V(\cH)}b_v=k|\cB_X|\le kt\binom{k}{2}\lambda.
\end{equation}

Put $K:=k(k-1)$.  Fix $v$ and condition further on $\tau_v=x$ and $A_v$.  Consider an edge $f$ containing $v$ with $f\ne X_v$ and $f\notin\cB_X$.  For each $u\in f\setminus\{v\}$, let $X_u$ be the edge of the fixed matching $X$ containing $u$.  Since $f\notin\cB_X$, these $k-1$ edges are distinct and they are also different from $X_v$.  Their union therefore consists of exactly $K=k(k-1)$ vertices, all outside $X_v$.

If any one of these $K$ vertices has priority larger than $x$, then the corresponding edge $X_u$ is revealed before $X_v$.  That revealed edge meets $f$ but is not equal to $f$, so no perfect matching agreeing with the variables already revealed can contain $f$.  Hence $f$ can remain a possible value of $X_v$, that is, $f$ may be counted in $N_v$, only if all these $K$ priorities are below $x$.  The event $A_v$ depends only on the priorities in $X_v\setminus\{v\}$, which are disjoint from these $K$ vertices.  Thus conditioning on $A_v$ does not change their independent uniform distributions, and
\[
 \PP(f\text{ remains possible}\mid X,\tau_v=x,A_v)\le x^K.
\]

For every such edge $f$, let $I_f$ be the indicator that $f$ remains a possible value of $X_v$ after the earlier variables have been revealed.  The actual edge $X_v$ contributes one possible value, and the edges of $\cB_X$ through $v$ contribute at most $b_v$ additional values.  Consequently,
\[
 N_v\le 1+b_v+
 \sum_{\substack{f\ni v\\ f\ne X_v,\ f\notin\cB_X}} I_f.
\]
Taking conditional expectations and using linearity of expectation, together with the preceding probability bound and $d_{\cH}(v)\le D$, gives
\begin{equation}\label{eq:support-expectation}
 \EE_\tau[N_v\mid X,\tau_v=x,A_v]\le 1+b_v+Dx^K.
\end{equation}

For each fixed $X$ and $x$, Jensen's inequality applied to the sum over $v$ gives
\[
\begin{aligned}
 \sum_{v\in V(\cH)}\log(1+b_v+Dx^K)
 &\le kt\log\left(\frac1{kt}\sum_{v\in V(\cH)}(1+b_v+Dx^K)\right)\\
 &=kt\log\left(1+\frac1{kt}\sum_{v\in V(\cH)}b_v+Dx^K\right)\\
 &\le kt\log\left(1+\binom{k}{2}\lambda+Dx^K\right),
\end{aligned}
\]
where the last inequality is due to \eqref{eq:bad-incidences}. Note the final expression is independent of the sampled matching $X$. Therefore, substituting \eqref{eq:support-expectation} into \eqref{eq:entropy-integral} gives 
\[
 \HH(X)\le \EE_X\int_0^1 x^{k-1}
 \sum_{v\in V(\cH)}\log(1+b_v+Dx^K) \,dx \le \int_0^1x^{k-1}\cdot
 kt\log\left(1+\binom{k}{2}\lambda+Dx^K\right)\,dx.
\]
Since $1+\binom{k}{2}\lambda+Dx^K=D(\rho+x^K)$ and $kt\int_0^1x^{k-1}\,dx=t$, this becomes
\begin{equation}\label{eq:entropy-J}
 \HH(X)\le t\log D+ktJ(\rho),
 \qquad
 J(\rho):=\int_0^1x^{k-1}\log(x^K+\rho)\,dx.
\end{equation}

It remains to estimate $J(\rho)$.  Since $K=k(k-1)$ and
$\int_0^1x^{k-1}\log x\,dx=-1/k^2$, we have $J(0)=-(k-1)/k$.
For $\rho>0$, the substitution $x=\rho^{1/K}y$ gives
\[
 \begin{aligned}
 0\le J(\rho)-J(0)
 &=\int_0^1x^{k-1}\log(1+\rho x^{-K})\,dx\\
 &=\rho^{1/(k-1)}
   \int_0^{\rho^{-1/K}}y^{k-1}\log(1+y^{-K})\,dy\\
 &\le\rho^{1/(k-1)}
   \int_0^\infty y^{k-1}\log(1+y^{-K})\,dy
 =O_k(\rho^{1/(k-1)}).
 \end{aligned}
\]
The last integral is finite for $k\ge3$: its integrand is
$O_k(y^{k-1}|\log y|)$ near zero and
$O_k(y^{-(k-1)^2})$ at infinity.  Thus
\[
 J(\rho)=-\frac{k-1}{k}+O_k(\rho^{1/(k-1)}).
\]
Substituting this estimate into \eqref{eq:entropy-J} gives
\[
 \HH(X)\le t\left(\log D-(k-1)
 +O_k\bigl(\rho^{1/(k-1)}\bigr)\right).
\]
Since $\HH(X)=\log M_t(\cH)$, choosing $C_k$ to dominate the implicit constant proves \eqref{eq:entropy-perfect-bound}, including for multi-hypergraphs.

\medskip
\noindent\textbf{General matching sizes.}
Now let $\cH$ be the (simple) hypergraph given by the theorem.  There is nothing to prove if $M_t(\cH)=0$, and the case $\alpha=1$ follows from \eqref{eq:entropy-perfect-bound}.  Suppose that $0<\alpha<1$.  Keep $\cH,N,t,D,\lambda$ fixed, and let $\ell$ tend to infinity through positive multiples of $k$.  Take the disjoint union $\cH^{(\ell)}$ of $\ell$ copies of $\cH$, and write
\[
 v=\ell N,\qquad u=v/k,\qquad r=\ell t=\alpha u,
 \qquad s=u-r=(1-\alpha)u.
\]
These quantities are positive integers.  Observe that $M_r(\cH^{(\ell)})\ge M_t(\cH)^\ell$.
Put $D_0:=\binom{v-1}{k-1}$ and $\lambda_0:=\binom{v-2}{k-2}$.
Define a multi-hypergraph $\tilde{\cH}$ on $V:=V(\cH^{(\ell)})$ by assigning an \emph{ordinary} edge to every $k$-subset of $V$ and $y$ \emph{distinguished} edges to every edge of $\cH^{(\ell)}$, where 
\[
 y:=\frac{rD_0}{sD} =\frac{\alpha D_0}{(1-\alpha)D}.
\]
Let $Z_\ell(y)$ be the number of perfect matchings in $\tilde{\cH}$. Fixing one $t$-edge matching in each original copy and choosing distinguished copies for its edges gives $M_t(\cH)^\ell y^r$ possibilities.  The remaining $ks$ vertices can be partitioned into $s$ ordinary edges in $Q_s:=\frac{(ks)!}{(k!)^s s!}$
ways.  The distinguished and ordinary copies are different even when they have the same vertices, so the construction is injective.  Consequently,
\begin{equation}\label{eq:entropy-completion-count}
 Z_\ell(y)\ge M_t(\cH)^\ell y^r Q_s.
\end{equation}

The maximum degree and maximum 2-degree of $\tilde{\cH}$ are at most $D_0+yD$ and $\lambda_0+y\lambda$, respectively.  As $\ell\to\infty$, we have \(\frac{\lambda_0}{D_0}=\frac{k-1}{v-1}\rightarrow0\).
Consequently,
\[
 \rho_\ell:=
 \frac{1+\binom{k}{2}(\lambda_0+y\lambda)}{D_0+yD}
 \rightarrow
 \binom{k}{2}\frac{\alpha\lambda}{D}<\rho,
\]
so $\rho_\ell\le\rho\le\rho_k$ for all sufficiently large $\ell$.  Applying the perfect-matching estimate \eqref{eq:entropy-perfect-bound} to $\tilde{\cH}$ gives
\[
 Z_\ell(y)
 \le
 \left(\frac{D_0+yD}{\e^{k-1}}\right)^u
 \exp\left(C_ku\rho^{1/(k-1)}\right).
\]
Combining this with \eqref{eq:entropy-completion-count}, taking $\ell$-th roots, and using $u/ \ell = N/k$, $r/\ell=t$, and $s/\ell=(N-kt)/k$, we obtain
\begin{equation}\label{eq:entropy-coefficient-general}
 M_t(\cH) \le
 \left(\frac{D_0+yD}{y\e^{k-1}}\right)^t
 \left(\frac{D_0+yD}{\e^{k-1}Q_s^{1/s}}\right)^{(N-kt)/k} \exp\left(\frac{C_kN}{k}\rho^{1/(k-1)}\right).
\end{equation}
The outer exponents and the exponential factor are independent of $\ell$, so it remains to evaluate the limits of the first two bases.
For the first base,
\[
 \frac{D_0+yD}{y\e^{k-1}}
 =\frac{D_0/y+D}{\e^{k-1}}
 \longrightarrow
 \frac{(1-\alpha)D/\alpha+D}{\e^{k-1}}
 =\frac{D}{\alpha\e^{k-1}}.
\]
For the second base, Stirling's formula gives
\[
 Q_s^{1/s}
 =\left(\frac{(ks)!}{(k!)^s s!}\right)^{1/s}
 =(1+o(1))\frac{(ks)^{k-1}}{(k-1)!\,\e^{k-1}}.
\]
Since $ks=(1-\alpha)v$ and $D_0=(1+o(1))v^{k-1}/(k-1)!$, it follows that
\[
 \frac{D_0+yD}{\e^{k-1}Q_s^{1/s}}
 =\left(1+\frac{yD}{D_0}\right)
   \frac{D_0}{\e^{k-1}Q_s^{1/s}}
 \longrightarrow
 \frac1{1-\alpha}\cdot\frac1{(1-\alpha)^{k-1}}
 =\frac1{(1-\alpha)^k}.
\]
Letting $\ell\to\infty$ in \eqref{eq:entropy-coefficient-general}, we therefore obtain
\[
 M_t(\cH)
 \le
 \left(\frac{D}{\alpha\e^{k-1}}\right)^t
 (1-\alpha)^{-(N-kt)}
 \exp\left(\frac{C_kN}{k}\rho^{1/(k-1)}\right),
\]
which is \eqref{eq:entropy-general-bound}.
\end{proof}

\section{Latin squares}\label{sec:latin}

Let $L$ be a Latin square of order $n$, with row set $A$, column set $B$, and symbol set $C$.  Define a properly edge-colored copy $G(L)$ of $K_{n,n}$ on $A\cup B$ by coloring the edge $ab$ with the symbol in cell $(a,b)$.  Its associated hypergraph $\cH(L):=\cH(G(L))$ has vertex classes $A,B,C$.  Every vertex of $\cH(L)$ has degree $n$, and $\cH(L)$ is linear.  A set of cells is a partial transversal of $L$ if and only if the corresponding edges form a rainbow matching in $G(L)$, or equivalently the corresponding hyperedges form a matching in $\cH(L)$.

\begin{proof}[Proof of Theorem~\ref{thm:latin-main}]
\emph{Lower bound.}
Fix a sufficiently small constant $0<a<1$, and choose $\eta$ and sufficiently large $n$ so that $1/n\llpoly\eta\llpoly n^{-a}$.  The graph $G(L)$ is a properly colored copy of $K_{n,n}$ with exactly $n$ colors, and every color appears on exactly $n$ edges.  Proposition~\ref{prop:mont-pseudorandom} therefore implies that $G(L)$ is $(n,1,n^{-a})$-properly pseudorandom. 

We now apply Theorem~\ref{thm:lowerbound} with $M=n$ and $p_0=p=1$.  Denote the constants supplied by the theorem by $\sigma,C>0$; they are independent of $n$ and $L$.
The theorem gives
\[
 T_{n-1}(L)\ge
 \left(\frac n{\e^2}\right)^n\exp(-Cn^{1-\sigma}).
\]

\emph{Upper bound.}
Apply Theorem~\ref{thm:entropy} to $\cH(L)$ with $N=3n$, $k=3$, $t=n-1$, $D=n$, and $\lambda=1$.  Thus $\alpha=1-1/n$ and $\rho=4/n\le\rho_3$ for sufficiently large $n$, and
\[
 T_{n-1}(L)\le
 \left(\frac n{\alpha\e^2}\right)^{n-1}
 n^3\exp(2C_3\sqrt n).
\]
Since $\alpha^{-(n-1)} =\left(1+\frac1{n-1}\right)^{n-1}\le\e$,
we obtain
\[
 \begin{aligned}
 T_{n-1}(L)
 &\le\e^3n^2
 \left(\frac n{\e^2}\right)^n\exp(2C_3\sqrt n)\\
 &=\left(\frac n{\e^2}\right)^n\exp(O(\sqrt n)).
 \end{aligned}
\]
Taking $n$-th roots in the lower and upper bounds gives
\[
 \frac n{\e^2}\exp(-Cn^{-\sigma})
 \le T_{n-1}(L)^{1/n}
 \le\frac n{\e^2}\exp(O(n^{-1/2})).
\]
Fix $0<c<\min\{\sigma,1/2\}$.  Using $\e^{-x}\ge1-x$ for $x\ge0$ and $\e^x\le1+2x$ for $0\le x\le1$, the two exponential factors are bounded below by $1-n^{-c}$ and above by $1+n^{-c}$, respectively, for sufficiently large $n$.  Raising to the $n$-th power proves the theorem.
\end{proof}
\section{Steiner triple systems}\label{sec:sts}

Let $S$ be a Steiner triple system of order $n$, and put $m=\lfloor n/3\rfloor$.  Since $n\equiv1,3\pmod6$, either $n=3m$ or $n=3m+1$.  If $n=3m$, set $X:=\varnothing$; if $n=3m+1$, fix an arbitrary vertex $x\in V(S)$ and set $X:=\{x\}$.  In either case, put
\begin{align}\label{eq:V0}
     V^\circ:=V(S)\setminus X, \qquad |V^\circ|=3m.
\end{align}

For an ordered tripartition $\Pi=(A,B,C)$ of $V^\circ$, define the properly edge-colored bipartite graph $G_\Pi$ with classes $A,B$ by placing an edge $ab$ of color $c\in C$ precisely when $\{a,b,c\}\in E(S)$.  Thus rainbow matchings in $G_\Pi$ correspond to matchings of $S$ that avoid $X$ and whose edges meet each part of $\Pi$ once. 

Montgomery’s proof of \cite[Theorem~1.6]{MontgomeryRBS} shows that if $\Pi=(A, B, C)$ is random, then,  with positive probability, $G_\Pi$ is $(m, 1/3, \varepsilon)$-properly pseudorandom for some $\varepsilon>0$.\footnote{It refers to \cite[Section 6]{KPSY} for the proofs of \textnormal{(F1)} and \textnormal{(F2)}.} 
Furthermore, Properties~\textnormal{(F2)}--\textnormal{(F7)} hold with high probability -- we record this in the following lemma.

\begin{lemma}[{\cite[proof of Theorem~1.6, pp.~63--68]
{MontgomeryRBS}}]
\label{lem:pseudorandom-random-tripartitions}
Let $0<\varepsilon\le1$ satisfy $\varepsilon\ge(3m)^{-1/8}$. 
Let $S$ be a Steiner triple system of order $n$ with $V^\circ$ defined in \eqref{eq:V0}.
Form an ordered tripartition $\Pi=(A,B,C)$ of $V^\circ$ by
assigning each vertex independently and uniformly to $A,B,C$.
Then, the probability that $G_\Pi$ fails at least one of
 \textnormal{(F2)}--\textnormal{(F7)} of Definition~\ref{def:proper-pseudorandom},
 with parameters $(m,1/3,\varepsilon)$, is $o(n^{-3})$.
\end{lemma}


\begin{proof}[Proof of Theorem~\ref{thm:sts-main}]
\emph{Lower bound.}
Fix a sufficiently small constant $0<a\le1/8$ and let $\varepsilon=m^{-a}$, so that $\varepsilon\ge(3m)^{-1/8}$.
Assign the vertices of $V^\circ$ independently and uniformly to sets $A,B,C$, let $\Pi=(A,B,C)$ be the resulting tripartition.
Let $\mathcal B$ be the event that $G_\Pi$ fails at least one of
Properties~\textnormal{(F2)}--\textnormal{(F7)} of
Definition~\ref{def:proper-pseudorandom}.
Lemma~\ref{lem:pseudorandom-random-tripartitions} gives $\PP(\mathcal B)=o(n^{-3})$. 

Let $\mathcal E$ be the event that $|A|=|B|=|C|=m$. Let $\mathcal D$ be the event that every vertex in $C$ is a color of $G_\Pi$. If $\mathcal E$ and $\mathcal D$ hold, then $|A|=|B|=|C(G_\Pi)|=m$, so Property \textnormal{(F1)} holds.
Before conditioning on \(\mathcal E\), fix \(z\in V^\circ\) and condition on \(z\in C\).
The vertex $z$ lies in $d_z:=(n-1)/2 - |X|$ edges of $S-X$ which, after removing $z$, become disjoint pairs of vertices. Let $Y_z$ be the number of these pairs whose endpoints receive the labels $A,B$ in either order.  Then $Y_z\sim \operatorname{Bin}(d_z, 2/9)$ 
and $z$ is a color of $G_\Pi$ exactly when $Y_z>0$.  Hence
\[
\PP\bigl(z\in C\text{ and }Y_z=0\bigr) =\frac{1}{3}\left(\frac{7}{9}\right)^{d_z}
\le \frac{1}{3}\left(\frac{7}{9}\right)^{(n-3)/2}
\]
for all $z\in V^\circ$, and a union bound gives $\PP(\mathcal D^c)=o(n^{-3})$.
Stirling’s formula also gives
\[
 \PP(\mathcal E) =  \frac{(3m)!}{(m!)^3 3^{3m}} = \Theta(m^{-1}).
\]
Consequently,
\[
 \PP\!\left(
   G_\Pi\text{ is not $(m,1/3,m^{-a})$-properly pseudorandom}
   \,\middle|\,\mathcal E
 \right)
 \le
 \frac{\PP(\mathcal B)+\PP(\mathcal D^c)}{\PP(\mathcal E)}
 =
 o(m^{-2}).
\]
Conditioning on $\mathcal E$, the tripartition $\Pi$ is a uniformly random ordered equipartition of $V^\circ$.  Therefore a proportion $1-o(m^{-2})$ of these equipartitions give an $(m,1/3,m^{-a})$-properly pseudorandom graph $G_\Pi$.

Theorem~\ref{thm:lowerbound} applies to every such graph $G_\Pi$ with $M=m$ and $p_0=p=1/3$.  Denote the constants supplied by the theorem by $\sigma,C>0$.  Since $a$ is fixed, these constants are independent of $m$, $S$, and $\Pi$.  Decrease $\sigma$ if necessary so that $0<\sigma<1$.  Each graph $G_\Pi$ has at least $\Lambda_m$ rainbow matchings with $m-1$ edges, where
\begin{equation}\label{eq:Lambda-m}
 \Lambda_m:=\left(\frac{m}{3\e^2}\right)^m\exp(-Cm^{1-\sigma}).
\end{equation}

Let $\mathcal P$ be the set of ordered equipartitions of $V^\circ$, and let $\mathcal Q$ consist of pairs $(\Pi,R)$ such that $\Pi\in\mathcal P$, the graph $G_\Pi$ is $(m,1/3,m^{-a})$-properly pseudorandom, and $R$ is an $(m-1)$-edge matching of $S$ that avoids $X$ and whose edges meet each part of $\Pi$ once.  The preceding proportion estimate and \eqref{eq:Lambda-m} give
\[
 |\mathcal Q|
 \ge\bigl(1-o(m^{-2})\bigr)
 \frac{(3m)!}{(m!)^3}\Lambda_m.
\]

Fix an $(m-1)$-edge matching $R$ occurring in a pair in $\mathcal Q$.  It avoids $X$ and leaves exactly three vertices of $V^\circ$ uncovered.  The vertices of each edge of $R$ can be assigned bijectively to $A,B,C$ in $6$ ways, and the three uncovered vertices can also be assigned one to each part in $6$ ways.  Hence $R$ crosses exactly $6^{m-1}\cdot6=6^m$ ordered equipartitions of $V^\circ$.  The pseudorandomness requirement can only reduce this multiplicity.  Since the matchings that occur in some pair of $\mathcal Q$ form a subfamily of all $(m-1)$-edge matchings of $S$, we obtain
\[
 N_{m-1}(S)
 \ge\frac{|\mathcal Q|}{6^m}
 \ge\bigl(1-o(m^{-2})\bigr)
 \frac{(3m)!}{(m!)^3 6^m}\Lambda_m.
\]
Here, when $n=3m+1$, the argument counts only matchings avoiding the fixed vertex $x$, which is sufficient for a lower bound on $N_{m-1}(S)$.

For sufficiently large $m$, we have $1-o(m^{-2})\ge1/2$, and Stirling's formula gives
\[
 \frac{(3m)!}{(m!)^3 6^m}
 \ge2m^{-2}\left(\frac92\right)^m.
\]
Together with \eqref{eq:Lambda-m}, these estimates imply
\[
 N_{m-1}(S)\ge
 m^{-2}\left(\frac{3m}{2\e^2}\right)^m
 \exp(-Cm^{1-\sigma}).
\]
Since $3m\le n\le3m+1$,
\[
 \left(\frac{3m}{n}\right)^m
 \ge\left(1+\frac1{3m}\right)^{-m}\ge\e^{-1/3}.
\]
Consequently,
\[
 N_{m-1}(S)\ge
 \left(\frac n{2\e^2}\right)^m
 m^{-2}\exp(-Cm^{1-\sigma}-1/3)=
 \left(\frac n{2\e^2}\right)^m
 \exp\bigl(-O(m^{1-\sigma})\bigr).
\]

\emph{Upper bound.}
Every vertex of $S$ has degree $D=(n-1)/2$, and $\Delta_2(S)=1$.  Put $t=m-1$, $s=n-3t\in\{3,4\}$, and $\alpha=3t/n=1-s/n$.  Theorem~\ref{thm:entropy}, applied with $k=3$ and $\lambda=1$, has parameter $\rho=4/D=8/(n-1)\le\rho_3$ for sufficiently large $n$.  It gives
\[
 N_{m-1}(S)\le
 \left(\frac{D}{\alpha\e^2}\right)^t
 (1-\alpha)^{-s}
 \exp\left(\frac{C_3n}{3}\sqrt{\frac4D}\right).
\]
We have
\[
 \alpha^{-t}
 =\left(1+\frac{s}{n-s}\right)^t
 \le\e^{ts/(n-s)}
 =\e^{s/3}\le\e^{4/3},
 \qquad
 (1-\alpha)^{-s}=\left(\frac ns\right)^s\le n^4.
\]
Also, $n/4\le D\le n/2$ and $n/(2\e^2)>1$ for large $n$, so $(D/\e^2)^t\le(n/(2\e^2))^m$.  Therefore
\[
 \begin{aligned}
 N_{m-1}(S)
 &\le
 \left(\frac n{2\e^2}\right)^m
 \e^{4/3}n^4\exp\left(\frac{4C_3}{3}\sqrt n\right)\\
 &=
 \left(\frac n{2\e^2}\right)^m\exp(O(\sqrt n)).
 \end{aligned}
\]
Since $m=\Theta(n)$, taking $m$-th roots in the lower and upper bounds gives
\[
 \frac n{2\e^2}\exp(-O(n^{-\sigma}))
 \le N_{m-1}(S)^{1/m}
 \le\frac n{2\e^2}\exp(O(n^{-1/2})).
\]
Fix $0<c<\min\{\sigma,1/2\}$.  Using $\e^{-x}\ge1-x$ for $x\ge0$ and $\e^x\le1+2x$ for $0\le x\le1$, the two exponential factors are bounded below by $1-n^{-c}$ and above by $1+n^{-c}$, respectively, for sufficiently large $n$.  Raising to the $m$-th power proves the theorem.
\end{proof}

\section{Counting matchings of other sizes}\label{sec:all-sizes}

We prove Theorem~\ref{thm:matching-profiles} by counting one-edge extensions for small matchings, applying Theorem~\ref{thm:entropy} for upper bounds, and using Lemma~\ref{thm:gjkl} on random restrictions for large matchings.  Theorems~\ref{thm:latin-main} and~\ref{thm:sts-main} supply the lower bounds for the remaining almost-perfect sizes.

Throughout the section, $L$ and $S$ are arbitrary Latin squares and Steiner triple systems of order $n$, respectively, and $m=\lfloor n/3\rfloor$.
All subset and matching sizes in this section are integers.  
We use elementary factorial bounds
\begin{equation}\label{prof:stirling-uniform}
 \left(\frac{u}{\e}\right)^u
 \le u!
 \le \e\max\{1,u\}\left(\frac{u}{\e}\right)^u
 \qquad(u\ge0),
\end{equation}
where $0^0=1$.  Thus, when $u\le n$, replacing $u!$ by $(u/\e)^u$ costs at most a factor $\e n$.  All estimates below are uniform over the underlying structure and the matching size; constants in $O_\eta(\cdot)$ may depend on the fixed number $\eta>0$.

\subsection{Small matchings}

\begin{claim}[Small transversals]\label{prof:latin-extension}
For every Latin square $L$ of order $n$ and $0\le r\le n/3$,
\[
 T_r(L)=A_{n,r}\exp\left(\pm\frac{3r^3}{2n^2}\right).
\]
\end{claim}

\begin{proof}
Given $0\le j<n$, let $P$ be a $j$-transversal in $L$ and let $a(P)$ be the number of cells that extend it to a $(j+1)$-transversal.  Let $F_R,F_C,F_S$ be the sets of cells whose row, column, or symbol, respectively, is used by $P$.  Each has size $jn$, and each pairwise intersection has size $j^2$.  Write $q(P)=|F_R\cap F_C\cap F_S|\le j^2$.  A cell extends $P$ exactly when it lies outside $F_R\cup F_C\cup F_S$, so inclusion--exclusion gives
\[
 a(P)=n^2-3jn+3j^2-q(P).
\]
For $j<n/2$, this is at least $(n-j)(n-2j)>0$.  Starting with the empty transversal, we therefore have $T_j(L)>0$ for every $0\le j\le n/2$.

Define
\[
 b_j:=\frac{(j+1)A_{n,j+1}}{A_{n,j}}=\frac{(n-j)^3}{n},
 \qquad
 \varepsilon(P):=\frac{a(P)}{b_j}-1.
\]
Since $a(P)-b_j=j^3/n-q(P)$, we obtain
\begin{equation}\label{eq:epsP}
 -\left(\frac{j}{n-j}\right)^2
 \le\varepsilon(P)\le
 \left(\frac{j}{n-j}\right)^3.
\end{equation}
For $0\le j<r\le n/3$, let $\mathcal T_j$ be the family of all $j$-transversals and let $\varepsilon_j := \frac{1}{|\mathcal T_j|}\sum_{P\in\mathcal T_j}\varepsilon(P)$.  Counting pairs consisting of a $j$-transversal and an extending cell gives
\[
 (j+1)T_{j+1}(L)
 =\sum_{P\in\mathcal T_j}a(P)
 =b_jT_j(L)(1+\varepsilon_j).
\]
Dividing by $(j+1)A_{n,j+1}=b_jA_{n,j}$ and iterating from $T_0(L)=A_{n,0}=1$, we obtain
\[
 \frac{T_r(L)}{A_{n,r}}=\prod_{j=0}^{r-1}(1+\varepsilon_j).
\]
By \eqref{eq:epsP}, we have $|\varepsilon_j|\le(j/(n-j))^2\le1/4$.  Applying
$\e^{-2|x|}\le1+x\le\e^{2|x|}$ for $|x|\le1/2$ to each factor gives
\[
 \frac{T_r(L)}{A_{n,r}}
 =\exp\left(\pm2\sum_{j=0}^{r-1}
                  \left(\frac{j}{n-j}\right)^2\right).
\]
The claim follows because
\[
 2\sum_{j=0}^{r-1}\left(\frac{j}{n-j}\right)^2
 \le\frac{2}{(2n/3)^2}\sum_{j=0}^{r-1}j^2
 \le\frac{3r^3}{2n^2}.\qedhere
\]
\end{proof}

The corresponding estimate for Steiner triple systems is proved in Appendix~\ref{app:profile-proofs}.
\begin{claim}[Small matchings in Steiner triple systems]\label{prof:sts-extension}
For every Steiner triple system $S$ of order $n$ and every integer $0\le r\le m/3$,
\[
 N_r(S)=B_{n,r}\exp\left(\pm\frac{27r^3}{2n^2}\right).
\]
\end{claim}

\subsection{Upper bounds from the entropy theorem}\label{sec:dense-upper}

Theorem~\ref{thm:entropy} applies directly to the original hypergraphs and gives upper bounds throughout the possible matching-size ranges.

\begin{claim}[Upper bounds at all sizes]\label{prof:dense-upper}
For sufficiently large $n$ and every Latin square $L$ of order $n$ and every Steiner triple system $S$ of order $n$,
\begin{align*}
 T_r(L)&\le A_{n,r}\exp(O(\sqrt n))
 &&(1\le r\le n),\\
 N_r(S)&\le B_{n,r}\exp(O(\sqrt n))
 &&(1\le r\le m).
\end{align*}
The constants are absolute.
\end{claim}

\begin{proof}
\emph{Latin squares.}
Put $\alpha=r/n$ and $s=n-r$.  The hypergraph $\cH(L)$ has $3n$ vertices, degree $n$, and maximum 2-degree one.  Applying Theorem~\ref{thm:entropy} with $N=3n$, $k=3$, $t=r$, $D=n$, and $\lambda=1$, we obtain
\[
 T_r(L)\le
 \left(\frac{n}{\alpha\e^2}\right)^r
 (1-\alpha)^{-3s}\exp(O(\sqrt n)).
\]
On the other hand, \eqref{prof:stirling-uniform} gives
\[
 A_{n,r}
 =\frac{(n!)^3}{r!(s!)^3n^r}
 \ge
 \frac1{(\e n)^4}
 \left(\frac{n}{\alpha\e^2}\right)^r
 (1-\alpha)^{-3s}.
\]
Comparing these two inequalities proves the first bound.

\smallskip
\emph{Steiner triple systems.}
Put $D=(n-1)/2$, $\alpha=3r/n$, and $s=n-3r$.  Applying Theorem~\ref{thm:entropy} to $S$ with $N=n$, $k=3$, $t=r$, and $\lambda=1$, we obtain
\[
 N_r(S)\le
 \left(\frac{D}{\alpha\e^2}\right)^r
 (1-\alpha)^{-s}\exp(O(\sqrt n)).
\]
By \eqref{prof:stirling-uniform},
\[
 \begin{aligned}
 B_{n,r}
 &=\frac{n!}{s!\,6^r r!\,(n-2)^r}\\
 &\ge
 \frac1{(\e n)^2}
 \left(\frac{n^2}{(n-1)(n-2)}\right)^r
 \left(\frac{D}{\alpha\e^2}\right)^r
 (1-\alpha)^{-s}\\
 &\ge
 \frac1{(\e n)^2}
 \left(\frac{D}{\alpha\e^2}\right)^r
 (1-\alpha)^{-s}.
 \end{aligned}
\]
Comparing these two inequalities proves the second bound.
\end{proof}

\subsection{Lower bounds on large matchings}

Fix $\eta\in(0,1)$. Choose $0<\zeta<\min\{\zeta_0,1/4\}$, where $\zeta_0$ is supplied by Lemma~\ref{thm:gjkl} for $k=3$, and put $a=\zeta^3$.  We apply Lemma~\ref{thm:gjkl} to obtain lower bounds for $T_r(L)$ when $\eta n\le r\le\phi_L(n)$ and for $N_r(S)$ when $\eta m\le r\le\phi_S(n)$, where $\phi_L$ and $\phi_S$ are defined below and $\phi_L(n)=n-O(n^{1-a}),\phi_S(n)=m-O(n^{1-a})$.

We first need the following concentration estimates. Their proof is given in Appendix~\ref{app:profile-proofs}.

\begin{claim}[Degrees in random restrictions]\label{prof:restrictions}
Let $M$ be an integer with $\eta n/2\le M\le n$.
Choose independently and uniformly $M$ vertices in each class of $\cH(L)$, or a uniformly random $M$-subset of $V(S)$. 

We call the subhypergraphs of $\cH(L)$ or $S$ induced on the corresponding random sets \emph{(random) restrictions}. 
Put
\[
 \mu_L(M)=\frac{M^2}{n},\qquad
 \mu_S(M)=\frac{(M-1)(M-2)}{2(n-2)}.
\]
With probability $1-o(1)$, the restrictions have all their degrees within $n^{2/3}$ of $\mu_L(M)$ or $\mu_S(M)$, respectively.
\end{claim}

Define the deterministic degree bounds and matching sizes
\[
 \begin{aligned}
 d_L(M)&=\mu_L(M)+2n^{2/3},
 &\phi_L(M)&=(1-d_L(M)^{-a})M,\\
 d_S(M)&=\mu_S(M)+2n^{2/3},
 &\phi_S(M)&=(1-d_S(M)^{-a})M/3.
 \end{aligned}
\]

The next claim allows us to choose a restriction size for which Lemma~\ref{thm:gjkl} supplies matchings of exactly the desired size.  Its proof is also given in Appendix~\ref{app:profile-proofs}.

\begin{claim}\label{prof:restriction-tuning}
For sufficiently large $n$, every integer $\eta n\le r\le\phi_L(n)$ equals $\phi_L(M)$ for an integer $M\in[r,n]$ with $M-r=O_\eta(n^{1-a})$.  Every integer $\eta m\le r\le\phi_S(n)$ equals $\phi_S(M)$ for an integer $M\in[3r,n]$ with $M-3r=O_\eta(n^{1-a})$.
\end{claim}

We now apply Claims~\ref{prof:restrictions} and \ref{prof:restriction-tuning} and Lemma~\ref{thm:gjkl} to obtain the desired lower bounds.

\begin{claim}[Lower bounds on large matchings]\label{prof:dense-lower-away}
For sufficiently large $n$, every Latin square $L$ of order $n$ and every Steiner triple system $S$ of order $n$ satisfy
\begin{align*}
 T_r(L)&\ge A_{n,r}\exp\bigl(-O_\eta(n^{1-\zeta^4})\bigr)
 &&(\eta n\le r\le\phi_L(n)),\\
 N_r(S)&\ge B_{n,r}\exp\bigl(-O_\eta(n^{1-\zeta^4})\bigr)
 &&(\eta m\le r\le\phi_S(n)).
\end{align*}
\end{claim}

\begin{proof}
Let $n$ be sufficiently large, choose $M$ using Claim~\ref{prof:restriction-tuning}, and put $d=d_L(M)$ or $d_S(M)$, respectively.  Call a restriction from Claim~\ref{prof:restrictions} \emph{good} if every degree differs from the corresponding $\mu_L(M)$ or $\mu_S(M)$ by at most $n^{2/3}$.  Since $M\ge\eta n/2$ for sufficiently large $n$, we have $d=\Theta_\eta(n)$.  Every good restriction has degrees in $[d-3n^{2/3},d]$, and
\[
 3n^{2/3}=o_\eta(d^{1-\zeta}),\qquad
 \Delta_2\le1\le d^{1-\zeta},\qquad
 |V|\le3n\le\exp(d^{\zeta^3}).
\]
In particular, its minimum degree is at least $(1-d^{-\zeta})d$, and $d\ge d_0(3,\zeta)$ for large $n$.  Lemma~\ref{thm:gjkl} therefore supplies at least $((1-d^{-\zeta^4})d/\e^2)^r$ matchings of size exactly $r$ in each good restriction, by the definition of $\phi_L$ or $\phi_S$.

There are $\binom nM^3$ restrictions for $L$ and $\binom nM$ restrictions for $S$. By Claim~\ref{prof:restrictions}, a proportion $1-o(1)$ are good in each case.  A fixed $r$-transversal in $L$ lies in exactly $\binom{n-r}{M-r}^3$ restrictions; a fixed $r$-edge matching in $S$ lies in exactly $\binom{n-3r}{M-3r}$ restrictions.  Double counting pairs consisting of a good restriction and an $r$-transversal or an $r$-matching contained in it gives
\begin{equation}\label{prof:restriction-counts}
 \begin{aligned}
 T_r(L)&\ge(1-o(1))
 \left(\frac{\binom nM}{\binom{n-r}{M-r}}\right)^3
 \left(\frac{(1-d_L(M)^{-\zeta^4})d_L(M)}{\e^2}\right)^r,\\
 N_r(S)&\ge(1-o(1))
 \frac{\binom nM}{\binom{n-3r}{M-3r}}
 \left(\frac{(1-d_S(M)^{-\zeta^4})d_S(M)}{\e^2}\right)^r.
 \end{aligned}
\end{equation}
Note that the value of $M$ is chosen separately in these two inequalities.

We compare these expressions directly with $A_{n,r}$ and $B_{n,r}$.  The definitions give \(d_L(M)\ge\frac{M^2}{n}\), and \(d_S(M)\ge\frac{M^2}{2(n-2)}\).
For the second inequality, note that
$d_S(M)-\frac{M^2}{2(n-2)} =2n^{2/3}-\frac{3M-2}{2(n-2)}>0$
for large $n$ because $M\le n$.

\smallskip
\noindent\emph{Latin squares.}
Put $q=M-r\ge1$.  Using the definition of $A_{n,r}$ and then \eqref{prof:stirling-uniform}, we obtain
\[
\frac1{A_{n,r}}
 \left(\frac{\binom nM}{\binom{n-r}{M-r}}\right)^3
 \left(\frac{d_L(M)}{\e^2}\right)^r
 =\frac{r!(q!)^3}{(M!)^3}
 \left(\frac{nd_L(M)}{\e^2}\right)^r
 \ge
 \frac1{(\e n)^3}
 \left(\frac rM\right)^r
 \left(\frac qM\right)^{3q}.
\]
The displayed ratio is at least $\exp(-O(q\log n))$ because
\[
 \left(\frac rM\right)^r
 =\left(1+\frac qr\right)^{-r}\ge\e^{-q},
 \qquad
 \left(\frac qM\right)^{3q}\ge n^{-3q}.
\]

\smallskip
\noindent\emph{Steiner triple systems.}
Put $q=M-3r\ge1$. Using the definition of $B_{n,r}$ and then \eqref{prof:stirling-uniform} gives
\[
 \frac1{B_{n,r}}
 \frac{\binom nM}{\binom{n-3r}{M-3r}}
 \left(\frac{d_S(M)}{\e^2}\right)^r
 = \frac{q!r!}{M!}
 \left(\frac{6(n-2)d_S(M)}{\e^2}\right)^r
 \ge
 \frac1{\e n}
 \left(\frac{3r}{M}\right)^r
 \left(\frac qM\right)^q.
\]
The displayed ratio is at least $\exp(-O(q\log n))$ because
\[
 \left(\frac{3r}{M}\right)^r
 =\left(1+\frac q{3r}\right)^{-r}\ge\e^{-q/3},
 \qquad
 \left(\frac qM\right)^q\ge n^{-q}.
\]

In both cases, Claim~\ref{prof:restriction-tuning} gives $q=O_\eta(n^{1-a})$.  Since $a=\zeta^3>\zeta^4$, it follows that
\(q\log n=O_\eta(n^{1-a}\log n)=o_\eta(n^{1-\zeta^4})\).
Finally, $d=\Theta_\eta(n)$, $r\le n$, and $1-x\ge\e^{-2x}$ for $0\le x\le1/2$ imply
\[
 (1-o(1))(1-d^{-\zeta^4})^r
 \ge\exp\bigl(-O_\eta(n^{1-\zeta^4})\bigr).
\]
Combining these estimates with \eqref{prof:restriction-counts} proves both lower bounds.
\end{proof}

\subsection{Almost-perfect matchings}

The restriction argument stops a sublinear distance from the spanning size.  We cover the remaining sizes by counting submatchings of the matchings provided by Theorems~\ref{thm:latin-main} and~\ref{thm:sts-main}.

\begin{claim}\label{prof:endpoint-transfer}
For Latin squares, let \(1\le s\le n-1\); for Steiner triple systems, let \(1\le s\le m\). Then
\begin{equation}\label{prof:downward}
T_{n-s}(L) \ge\frac{\binom{n-1}{s-1}}{s^2(s-1)!}\,T_{n-1}(L),\qquad N_{m-s}(S)\ge\frac{\binom{m-1}{s-1}}{\binom{H_s}{s-1}}\,N_{m-1}(S),
\end{equation}
where $H_s=\left\lfloor\frac13\binom{n-3(m-s)}2\right\rfloor$.
For every fixed $b>0$, there is $c_b>0$ such that, for sufficiently large $n$ and every $1\le s\le n^{1-b}$,
\[
 \begin{aligned}
 T_{n-s}(L)\ge A_{n,n-s}\exp\bigl(-O_b(n^{1-c_b})\bigr),\qquad
 N_{m-s}(S)\ge B_{n,m-s}\exp\bigl(-O_b(n^{1-c_b})\bigr).
 \end{aligned}
\]
\end{claim}

\begin{proof}
We count pairs $(P,Q)$, where $Q$ is an $(n-1)$-transversal in $L$ or an $(m-1)$-matching in $S$, and $P\subseteq Q$ has $s-1$ fewer cells or edges.  For Latin squares, a fixed $Q$ has $\binom{n-1}{s-1}$ choices of $P$.  Conversely, an extension of $P$ uses $s-1$ of its $s$ unused rows and $s$ unused columns.  Ignoring symbols, choose those rows and columns and a bijection between them; there are at most $s^2(s-1)!$ possibilities.  This proves the first inequality.

For Steiner triple systems, a fixed $Q$ has $\binom{m-1}{s-1}$ choices of $P$.  The extension edges of $P$ lie on the $u_s=n-3(m-s)\in\{3s,3s+1\}$ vertices.  These span at most $H_s$ triples, since each triple uses three distinct pairs.  Ignoring disjointness, choose $s-1$ of these triples, giving the second inequality.

Choose a common constant $c_*\in(0,1)$ from Theorems~\ref{thm:latin-main} and~\ref{thm:sts-main}.  By \eqref{prof:stirling-uniform} and $n=3m+O(1)$,
\[
 A_{n,n-1}\le n^{O(1)}\left(\frac n{\e^2}\right)^n,
 \qquad
 B_{n,m-1}\le n^{O(1)}\left(\frac n{2\e^2}\right)^m.
\]
The two theorems and $1-x\ge\e^{-2x}$ for $0\le x\le1/2$ therefore give
\[
 \frac{T_{n-1}(L)}{A_{n,n-1}}
 \ge\exp\bigl(-O(n^{1-c_*})\bigr),
 \qquad
 \frac{N_{m-1}(S)}{B_{n,m-1}}
 \ge\exp\bigl(-O(n^{1-c_*})\bigr).
\]
By the definitions of $A_{n,r}$ and $B_{n,r}$, \eqref{prof:downward} gives
\[
 \frac{T_{n-s}(L)}{A_{n,n-s}} \ge\frac{s!}{n^{s-1}}
       \frac{T_{n-1}(L)}{A_{n,n-1}}, \qquad
 \frac{N_{m-s}(S)}{B_{n,m-s}}\ge n^{-3(s-1)}\frac{N_{m-1}(S)}{B_{n,m-1}}.
\]
For the second inequality, use $u_s!\ge u_1!$,
$(s-1)!\binom{H_s}{s-1}\le H_s^{s-1}$, and
$6(n-2)H_s\le n^3$.
Thus both normalized counts are at least
\(\exp\bigl(-O(n^{1-c_*}+s\log n)\bigr)\).
Since $s\le n^{1-b}$ implies $s\log n=O_b(n^{1-b/2})$, we may take $c_b=\min\{c_*,b/2\}$.
\end{proof}

\begin{proof}[Proof of Theorem~\ref{thm:matching-profiles}]
Claims~\ref{prof:latin-extension} and~\ref{prof:sts-extension} give the small-matching assertions with $C=27/2$.  Now fix $\eta\in(0,1)$.  Claim~\ref{prof:dense-upper} gives the upper bounds with a factor $\exp(O(\sqrt n))$, and Claim~\ref{prof:dense-lower-away} gives the lower bounds up to $\phi_L(n)$ and $\phi_S(n)$.

For the remaining Latin sizes, $s=n-r\le n-\phi_L(n)=O(n^{1-a})$; for the remaining Steiner sizes, $s=m-r\le m-\phi_S(n)=O(n^{1-a})$.  Taking $b=a/2$, both satisfy $s\le n^{1-b}$ for sufficiently large $n$.  Claim~\ref{prof:endpoint-transfer} supplies the remaining lower bounds.  Choose a common positive $c=\min\{\zeta^4,c_b\}$ and choose $C'$ and $n_0$ sufficiently large.  The upper and lower bounds above then give the required factors $\exp(\pm C'n^{1-c})$, proving the theorem.
\end{proof}

\section{Concluding remarks}\label{sec:conclusion}

We end with several related counting questions.

\subsection*{Counting matchings of maximum size}
The situation of counting $n$-cell transversals or $m$-edge matchings is less clear, since the number of them could be zero.  It is natural to ask how large the numbers must be once they exist.  More precisely, what are the minimum
$T_n(L)$ among all Latin squares $L$ of order $n$ with $T_n(L)>0$ and the minimum $N_m(S)$ among all Steiner triple systems $S$ of order $n$ with $N_m(S)>0$?

\subsection*{Decompositions of Latin squares}
Two Latin squares of order $n$ are \emph{orthogonal} if the pairs of entries in corresponding cells are distinct so that all possible $n^2$ pairs appear exactly once. Finding two orthogonal Latin squares is equivalent to finding a Latin square which can be decomposed into $n$ disjoint full transversals. Let
\(D(L)\) denote the number of such decompositions of any given Latin square $L$.

For every \(n>3\), Wanless and Webb~\cite{WanlessWebb} constructed a Latin square of order $n$ containing a cell that lies in no
transversal, and hence having no decomposition. On the other hand, results of Keevash and Luria~\cite{KeevashColoured,Luria} imply that the average \(D(L)\) over all order-\(n\) Latin squares $L$ is $\left(\left(1+ o\left(1\right)\right)n/e^3\right)^{n^2}$, and Boyadzhiyska, Das and Szab\'o~\cite{BDS} proved this asymptotic upper bound for every $L$. Bowtell and Montgomery~\cite{BowtellMontgomery} proved that a uniformly random Latin square \(L_n\) of order $n$ satisfies \(D(L_n)>0\) with high probability.  It is natural to ask whether $D(L_n)\ge \left(\left(1-o\left(1\right)\right)n/e^3\right)^{n^2}$ holds with high probability.

Motivated by our counting result, it would be interesting to know whether for sufficiently large $n$, for every Latin square $L$ of order $n$ and every cell $a$ of $L$, the cells other than $a$ can be partitioned into $n+1$ disjoint $(n-1)$-transversals. A weaker version is to ask whether such a cell $a$ exists for every $L$.

\subsection*{Resolutions of Steiner triple systems}

Suppose that \(n\equiv3\pmod6\) and \(d=(n-1)/2\). A \emph{resolution} of a Steiner triple system \(S\) of order \(n\) is a partition of its triples into \(d\) perfect matchings.  Let \(R(S)\) denote the number of such resolutions.

The results of Luria~\cite{Luria} and Dai, Divoux and Kelly~\cite{DaiDivouxKelly} imply that $R(S)\le \left(     (1+o(1))n/(2\e^3)  \right)^{n^2/6}$.
Ferber and Kwan~\cite{FerberKwan} proved that a uniformly random Steiner triple system $S_n$ of order $n$ is almost resolvable with high probability and conjectured that it is resolvable with high probability. These results suggest a counting question for resolutions: Does $R(S_n)\ge \left(     (1-o(1))n/(2\e^3)  \right)^{n^2/6}$ hold with high probability?

\smallskip
\subsection*{Declaration on the use of AI}
During the preparation of this work, the authors used ChatGPT 5.5 and 5.6 to assist with developing and checking
mathematical arguments and with the presentation. After using this tool, the authors reviewed and edited the content as needed and take full responsibility for the manuscript.

\appendix

\section{Typicality and proper pseudorandomness}
\label{app:montgomery-definitions}

We give Definitions~3.8--3.11 of Montgomery~\cite{MontgomeryRBS} below as Definitions~\ref{def:typical-bipartite}--\ref{def:proper-pseudorandom}.  
A matching is \emph{exactly-$D$-rainbow} if it uses every color in $D$ exactly once and no other colors.  For a positive integer $t$, write $[t]=\{1,\ldots,t\}$.

\begin{definition}\label{def:typical-bipartite}
A bipartite graph $H$ with vertex classes $A$ and $B$ is \emph{$(M,p,\varepsilon)$-typical} if the following conditions hold.
\begin{itemize}
\item $|A|=(1\pm\varepsilon)M$ and $|B|=(1\pm\varepsilon)M$.
\item For every $v\in V(H)$, $d_H(v)=(1\pm\varepsilon)pM$.
\item For each distinct $u,v\in V(H)$ with $u,v\in A$ or $u,v\in B$, we have $|N_H(u)\cap N_H(v)|=(1\pm\varepsilon)p^2M.$
\end{itemize}
\end{definition}

\begin{definition}\label{def:shadow-graph}
Given a $3$-uniform hypergraph $\cH$ and disjoint sets $X,Y\subseteq V(\cH)$, let $\cH_{X,Y}$ be the bipartite graph with vertex classes $X$ and $Y$ in which $xy$ is an edge precisely when there is some $z\in V(\cH)\setminus\{x,y\}$ such that $\{x,y,z\}\in E(\cH)$.
\end{definition}

\begin{definition}\label{def:typical-hypergraph}
A linear $3$-partite $3$-uniform hypergraph $\cH$ with vertex classes $A$, $B$, and $C$ is \emph{$(M,p,\varepsilon)$-typical} if each of $\cH_{A,B}$, $\cH_{B,C}$, and $\cH_{A,C}$ is $(M,p,\varepsilon)$-typical.
\end{definition}

\begin{definition}\label{def:proper-pseudorandom}
A bipartite graph $G$ with vertex classes $A$ and $B$ is \emph{$(M,p,\varepsilon)$-properly pseudorandom} if it is properly edge-colored and the following conditions hold with $\alpha=p^{12}/10^{100}$.
\begin{enumerate}
\item[\textnormal{(F1)}]
$|A|=|B|=M$ and $M\le |C(G)|\le(1+\varepsilon)M$.

\item[\textnormal{(F2)}]
$\cH(G)$ is $(M,p,\varepsilon)$-typical.

\item[\textnormal{(F3)}]
For each $c\in C(G)$ and $e\in E_c(G)$, for all but at most $\sqrt M$ edges $f\in E_c(G)\setminus\{e\}$, there are at least $\alpha M^2$ pairs $(S_1,S_2)$ such that $S_1$ and $S_2$ are vertex-disjoint rainbow $4$-cycles with $e\in E(S_1)$ and $f\in E(S_2)$, and the color sets of the two neighboring edges of $e$ in $S_1$ and the two neighboring edges of $f$ in $S_2$ are the same.

\item[\textnormal{(F4)}]
For each $u\in A$, $v\in B$, and $c_0\in C(G)$, there are pairwise disjoint sets $V_1,\ldots,V_{\alpha M}\subseteq V(G)\setminus\{u,v\}$
and pairwise disjoint sets $C_1,\ldots,C_{\alpha M}\subseteq C(G)\setminus\{c_0\}$ such that, for each $i\in[\alpha M]$, $|V_i|=4$, $|C_i|=3$, the graph $G[V_i]$ contains two color-$c_0$ edges and $G[\{u,v\}\cup V_i]$ contains an exactly-$C_i$-rainbow matching in $E(G)\setminus \{uv\}$.

\item[\textnormal{(F5)}]
For each distinct $c_0,d\in C(G)$, there are pairwise disjoint sets $V_1,\ldots,V_{\alpha M/12}\subseteq V(G)$ and pairwise disjoint sets $C_1,\ldots,C_{\alpha M/12}\subseteq C(G)\setminus\{c_0,d\}$
such that, for each $i\in[\alpha M/12]$, $|V_i|=8$, $|C_i|=3$, and $G[V_i]$ contains a matching of four color-$c_0$ edges and an exactly-$(C_i\cup\{d\})$-rainbow matching.

\item[\textnormal{(F6)}]
For any $c_0\in C(G)$, $0\le k\le20$, and any $\bar C\subseteq C(G)\setminus\{c_0\}$ with $|\bar C|\ge5k$, there are pairwise disjoint sets $\bar V_1,\ldots,\bar V_{\alpha M}\subseteq V(G)$ such that, for each $i\in[\alpha M]$, $|\bar V_i|=2k+2$ and $G[\bar V_i]$ contains both a matching of $k+1$ color-$c_0$ edges and a $\bar C$-rainbow matching with $k$ edges.

\item[\textnormal{(F7)}]
Set $k=100$.  For each $c_0\in C(G)$, there is some $r\in\mathbb N$ and pairwise disjoint sets $V_1,\ldots,V_r\subseteq V(G)$ and pairwise disjoint sets $C_1,\ldots,C_r\subseteq C(G)\setminus\{c_0\}$ with $|V_i|=2k$ and $|C_i|=k$ for each $i\in[r]$ such that the graph $G[V_i]$ contains an exactly-$C_i$-rainbow matching and a perfect matching of color-$c_0$ edges, and the following holds.

For every $\bar C\subseteq C(G)\setminus\{c_0\}$ with $|\bar C|\le k$, for at least $\alpha^2M$ values of $i\in[r]$, there are pairwise disjoint sets $\bar V_1,\ldots,\bar V_{\alpha M}\subseteq V(G)$ such that, for each $j\in[\alpha M]$, $|\bar V_j|=2k+2|\bar C|+2$ and $G[\bar V_j]$ contains both a matching of $k+|\bar C|+1$ color-$c_0$ edges and a $(\bar C\cup C_i)$-rainbow matching with $k+|\bar C|$ edges.
\end{enumerate}
\end{definition}

\section{Proof of Theorem~\ref{thm:reservoir}}
\label{app:reservoir-proof}

We first state two theorems of Montgomery \cite{MontgomeryRBS} and then apply them to prove Theorem~\ref{thm:reservoir}.  


The absorption theorem supplies, for each nonexceptional color $c$, a set
$F_c$ such that every sufficiently small matching $E\subseteq F_c$ has
an absorber for each submatching of the prescribed size.

\begin{lemma}[{\cite[Theorem~3.4]{MontgomeryRBS}}]\label{lem:mont-abs}
Let $M^{-1}\ll p,q_V,q_C\le1$ and $M^{-1}\llpoly\varepsilon\llpoly\gamma\llpoly\beta\llpoly\alpha\llpoly\log^{-1}M$.
Let $G$ be an $(M,p,\varepsilon)$-properly pseudorandom bipartite graph with classes $A,B$.  Choose $V\subseteq V(G)$ and $C\subseteq C(G)$ independently, including each element with probability $q_V$ and $q_C$, respectively.  With probability $1-o(1)$, all but at most $\alpha M$ colors $c\in C$ have a set $F_c\subseteq E_c(G)$ satisfying $|E_c(G[V])\setminus F_c|\le\alpha M$ and the following property.

For every $0\le\ell_0\le\ell_1\le\gamma M$, every set $E\subseteq F_c$ with $|E|=\ell_1$, and every $V'\subseteq V(G)$ and $C'\subseteq C(G)$ with $|V'|,|C'|\le10\gamma M$, there are sets $V_{\mathrm{abs}}\subseteq V\setminus\bigl(V(E)\cup V'\bigr)$ and $C_{\mathrm{abs}}\subseteq C\setminus C'$ such that $|A\cap V_{\mathrm{abs}}|=|B\cap V_{\mathrm{abs}}|=\beta M-\ell_0$ and $|C_{\mathrm{abs}}|=\beta M$.  For every $E'\subseteq E$ with $|E'|=\ell_0$, the graph $G[V_{\mathrm{abs}}\cup V(E')]$ contains a $C_{\mathrm{abs}}$-rainbow matching with $\beta M$ edges.
\end{lemma}

The addition theorem turns a small leftover into a matching of a specified
color and a vertex-disjoint rainbow matching, in the form needed by absorption.

\begin{lemma}[{\cite[Theorem~3.5]{MontgomeryRBS}}]\label{lem:mont-add}
Let $M^{-1}\ll p,q_V,q_C\le1$ and $M^{-1}\llpoly\varepsilon\llpoly\eta\llpoly\gamma\llpoly\alpha\llpoly\log^{-1}M$.
Let $G$ be an $(M,p,\varepsilon)$-properly pseudorandom bipartite graph with classes $A,B$.  Choose $V\subseteq V(G)$ and $C\subseteq C(G)$ independently, including each element with probability $q_V$ and $q_C$, respectively.  With probability $1-o(1)$, the following holds simultaneously for every $c\in C(G)$ and every $F\subseteq E_c(G)$ with $|E_c(G)\setminus F|\le\alpha M$.

There are sets $V_{\mathrm{add}}\subseteq V$ and $C_{\mathrm{add}}\subseteq C$ with $|A\cap V_{\mathrm{add}}|=|B\cap V_{\mathrm{add}}|=2\gamma M+1$ and $|C_{\mathrm{add}}|=\gamma M+1$, and they satisfy the following statement.  Whenever $\widehat A\subseteq A\setminus V_{\mathrm{add}}$, $\widehat B\subseteq B\setminus V_{\mathrm{add}}$ and $\widehat C\subseteq C(G)\setminus C_{\mathrm{add}}$ satisfy $|\widehat A|=|\widehat B|\le\eta M$ and $|\widehat C|\le\eta M$, the graph $G[V_{\mathrm{add}}\cup\widehat A\cup\widehat B]$ contains vertex-disjoint matchings $M_{\mathrm{id}}$ and $M_{\mathrm{rb}}$ such that
\begin{enumerate}
\item[\textnormal{(i)}] $M_{\mathrm{id}}$ consists of $\gamma M+|\widehat A|-|\widehat C|$ edges of $F$, all with both endpoints in $V_{\mathrm{add}}$;
\item[\textnormal{(ii)}] $M_{\mathrm{rb}}$ is $(C_{\mathrm{add}}\cup\widehat C)$-rainbow and has $|C_{\mathrm{add}}\cup\widehat C|-1$ edges.
\end{enumerate}
\end{lemma}


We now prove Theorem~\ref{thm:reservoir} by following \cite[Proof of Theorem~3.1, pp.~19--20]{MontgomeryRBS}.

\begin{proof}[Proof of Theorem~\ref{thm:reservoir}]
Choose auxiliary parameters $\gamma,\alpha$ so that $\varepsilon\llpoly\eta\llpoly\gamma\llpoly\beta\llpoly\alpha\llpoly\log^{-1}M$, and put $\beta_0:=\beta-\gamma$.  Thus $\gamma\llpoly\beta_0\llpoly\alpha$.
Since $|C_0|\sim\operatorname{Bin}(M,1/4)$ and $\alpha<1/8$ for all sufficiently large $M$, Chernoff's inequality gives
\[
 \PP(|C_0|\le\alpha M)
 \le \PP\bigl(\bigl||C_0|-M/4\bigr|\ge M/8\bigr)
 \le 2\exp(-M/48)=o(1).
\]
Apply Lemma~\ref{lem:mont-abs} with $3\gamma$ in place of $\gamma$ and $\beta_0$ in place of $\beta$, and apply Lemma~\ref{lem:mont-add} with parameter $\gamma$, taking $V=V_0,C=C_0$, $q_V=q_C=1/4$ in both lemmas.  Then $|C_0|>\alpha M$ and the following two properties hold simultaneously with probability $1-o(1)$.

\begin{enumerate}
\item[\textnormal{(P1)}] For all but at most $\alpha M$ colors $c\in C_0$, there is a set $F_c\subseteq E_c(G)$ with $|E_c(G[V_0])\setminus F_c|\le\alpha M$.  For every such color, the set $F_c$ has the following property.
  \begin{enumerate}
  \item[\textnormal{(i)}] If $0\le\ell_0\le\ell_1\le3\gamma M$, $E\subseteq F_c$ has size $\ell_1$, and $V'\subseteq V(G)$ and $C'\subseteq C(G)$ satisfy $|V'|,|C'|\le30\gamma M$, then there are sets $V_{\mathrm{abs}}\subseteq V_0\setminus\bigl(V(E)\cup V'\bigr)$ and $C_{\mathrm{abs}}\subseteq C_0\setminus C'$ with $|A\cap V_{\mathrm{abs}}|=|B\cap V_{\mathrm{abs}}|=\beta_0 M-\ell_0$ and $|C_{\mathrm{abs}}|=\beta_0 M$ such that, for every $E'\subseteq E$ of size $\ell_0$, the graph $G[V_{\mathrm{abs}}\cup V(E')]$ has a $C_{\mathrm{abs}}$-rainbow matching with $\beta_0 M$ edges.
  \end{enumerate}

\item[\textnormal{(P2)}] For every $c\in C(G)$ and every $E_0\subseteq E_c(G)$ with $|E_c(G)\setminus E_0|\le\alpha M$, there are sets $V_{\mathrm{add}}\subseteq V_0$ and $C_{\mathrm{add}}\subseteq C_0$ with the following properties.
  \begin{enumerate}
  \item[\textnormal{(ii)}] We have $|A\cap V_{\mathrm{add}}|=|B\cap V_{\mathrm{add}}|=2\gamma M+1$ and $|C_{\mathrm{add}}|=\gamma M+1$.
  \item[\textnormal{(iii)}] Whenever $\widehat A\subseteq A\setminus V_{\mathrm{add}}$, $\widehat B\subseteq B\setminus V_{\mathrm{add}}$, and $\widehat C\subseteq C(G)\setminus C_{\mathrm{add}}$ satisfy $|\widehat A|=|\widehat B|\le\eta M$ and $|\widehat C|\le\eta M$, the graph $G[V_{\mathrm{add}}\cup\widehat A\cup\widehat B]$ contains vertex-disjoint matchings $M_{\mathrm{id}}$ and $M_{\mathrm{rb}}$ such that $M_{\mathrm{id}}\subseteq E_0$, $V(M_{\mathrm{id}})\subseteq V_{\mathrm{add}}$, and $|M_{\mathrm{id}}|=\gamma M+|\widehat A|-|\widehat C|$, while $M_{\mathrm{rb}}$ is $(C_{\mathrm{add}}\cup\widehat C)$-rainbow and has $|C_{\mathrm{add}}\cup\widehat C|-1$ edges.
  \end{enumerate}
\end{enumerate}

Since $|C_0|>\alpha M$, choose $c_0\in C_0$ and a corresponding set $F_{c_0}$ as in \textnormal{(P1)}.  Put $E_0:=F_{c_0}\cup\bigl(E_{c_0}(G)\setminus E(G[V_0])\bigr)$.  Then $E_{c_0}(G)\setminus E_0=E_{c_0}(G[V_0])\setminus F_{c_0}$ has size at most $\alpha M$, so \textnormal{(P2)} gives sets $V_{\mathrm{add}},C_{\mathrm{add}}$ satisfying \textnormal{(ii)} and \textnormal{(iii)}.

Apply \textnormal{(iii)} with $\widehat A=\widehat B=\widehat C=\varnothing$.  The set $E_0$ contains a matching of size $\gamma M$ whose endpoints all lie in $V_{\mathrm{add}}$.  Let $E_1$ be the set of all edges in $E_0$ whose two endpoints lie in $V_{\mathrm{add}}$.  Then $|E_1|\ge\gamma M$.  Since $E_1\subseteq E_{c_0}(G)$ and every color class is a matching, $E_1$ is a matching in $G[V_{\mathrm{add}}]$.  Hence, by \textnormal{(ii)}, $|E_1|\le |A\cap V_{\mathrm{add}}|=2\gamma M+1$.  Since $V_{\mathrm{add}}\subseteq V_0$, every edge of $E_1$ lies in $E(G[V_0])$.  The definition of $E_0$ therefore gives $E_1\subseteq F_{c_0}$.  Thus $\gamma M\le |E_1|\le2\gamma M+1\le3\gamma M$ for sufficiently large $M$.

Take $E=E_1$, $\ell_0=\gamma M$, $\ell_1=|E_1|$, $V'=V_{\mathrm{add}}$, and $C'=C_{\mathrm{add}}$.  The inequality $\gamma M\le |E_1|\le3\gamma M$ gives $\ell_0\le\ell_1\le3\gamma M$, while \textnormal{(ii)} gives $|V_{\mathrm{add}}|=4\gamma M+2\le30\gamma M$ and $|C_{\mathrm{add}}|=\gamma M+1\le30\gamma M$.  Hence \textnormal{(P1)(i)} gives sets $V_{\mathrm{abs}}\subseteq V_0\setminus\bigl(V(E_1)\cup V_{\mathrm{add}}\bigr)$ and $C_{\mathrm{abs}}\subseteq C_0\setminus C_{\mathrm{add}}$ such that $|A\cap V_{\mathrm{abs}}|=|B\cap V_{\mathrm{abs}}|=(\beta-2\gamma)M$ and $|C_{\mathrm{abs}}|=(\beta-\gamma)M$.  These sets have the following property:

\begin{enumerate}
\item[\textnormal{(iv)}] For every matching $M'\subseteq E_1$ of size $\gamma M$, the graph $G[V_{\mathrm{abs}}\cup V(M')]$ has a $C_{\mathrm{abs}}$-rainbow matching with $(\beta-\gamma)M$ edges.
\end{enumerate}

Set $U_A:=A\cap(V_{\mathrm{add}}\cup V_{\mathrm{abs}})$, $U_B:=B\cap(V_{\mathrm{add}}\cup V_{\mathrm{abs}})$, and $U_C:=C_{\mathrm{add}}\cup C_{\mathrm{abs}}$.  We have $V_{\mathrm{add}}\subseteq V_0$ and $V_{\mathrm{abs}}\subseteq V_0\setminus V_{\mathrm{add}}$, while $C_{\mathrm{add}}\subseteq C_0$ and $C_{\mathrm{abs}}\subseteq C_0\setminus C_{\mathrm{add}}$.  Hence $U_A\subseteq A\cap V_0$, $U_B\subseteq B\cap V_0$, and $U_C\subseteq C_0$.  The size formulas above show that $|U_A|=|U_B|=|U_C|=\beta M+1$.

Fix a matching $P$ as in the statement of Theorem~\ref{thm:reservoir}.  Let $\widehat A,\widehat B,\widehat C$ be the elements left uncovered by $P$ in $A\setminus U_A$, $B\setminus U_B$, and $ C(G)\setminus U_C$, respectively.  The three sets $A\setminus U_A$, $B\setminus U_B$, and $C(G)\setminus U_C$ all have size $M-\beta M-1$, and $P$ covers $|P|$ elements in each.  Thus $\widehat A,\widehat B,\widehat C$ have the same size $\ell=M-\beta M-1-|P|\le\eta M$.  Statement \textnormal{(iii)} gives vertex-disjoint matchings $M_{\mathrm{id}}$ and $M_{\mathrm{rb}}$.  We have $M_{\mathrm{id}}\subseteq E_0$, $V(M_{\mathrm{id}})\subseteq V_{\mathrm{add}}$, and $|M_{\mathrm{id}}|=\gamma M$, while $M_{\mathrm{rb}}$ is $(C_{\mathrm{add}}\cup\widehat C)$-rainbow with $\gamma M+\ell$ edges.  Since $M_{\mathrm{id}}\subseteq E_0$ and $V(M_{\mathrm{id}})\subseteq V_{\mathrm{add}}$, the definition of $E_1$ gives $M_{\mathrm{id}}\subseteq E_1$.  Statement \textnormal{(iv)} therefore gives a $C_{\mathrm{abs}}$-rainbow matching $M_{\mathrm{abs}}$ of size $(\beta-\gamma)M$ in $G[V_{\mathrm{abs}}\cup V(M_{\mathrm{id}})]$.

Regard $P$ as the corresponding rainbow matching in $G$.  The matching $M_{\mathrm{rb}}$ has all its vertices in $V_{\mathrm{add}}\cup\widehat A\cup\widehat B$, whereas $M_{\mathrm{abs}}$ has all its vertices in $V_{\mathrm{abs}}\cup V(M_{\mathrm{id}})$.  The matching $P$ avoids $V_{\mathrm{add}}\cup V_{\mathrm{abs}}$ and leaves $\widehat A\cup\widehat B$ uncovered, so it is disjoint from both matchings.  By \textnormal{(iii)}, the matching $M_{\mathrm{rb}}$ is disjoint from $M_{\mathrm{id}}$.  We also have $V_{\mathrm{abs}}\cap V_{\mathrm{add}}=\varnothing$.  Since $V_{\mathrm{abs}}\subseteq U_A\cup U_B$, while $\widehat A\subseteq A\setminus U_A$ and $\widehat B\subseteq B\setminus U_B$, we have $V_{\mathrm{abs}}\cap(\widehat A\cup\widehat B)=\varnothing$.  Thus $M_{\mathrm{rb}}$ and $M_{\mathrm{abs}}$ are disjoint, and the three matchings are pairwise vertex-disjoint.

The matchings $P,M_{\mathrm{rb}},M_{\mathrm{abs}}$ use colors in $( C(G)\setminus U_C)\setminus\widehat C$, $C_{\mathrm{add}}\cup\widehat C$, and $C_{\mathrm{abs}}$, respectively.  These three sets are pairwise disjoint because $U_C=C_{\mathrm{add}}\cup C_{\mathrm{abs}}$ and $C_{\mathrm{add}}\cap C_{\mathrm{abs}}=\varnothing$.  Consequently their union is a rainbow matching in $G$ containing $P$, and its size is
\[
 \bigl(M-\beta M-1-\ell\bigr)+(\gamma M+\ell)+(\beta-\gamma)M=M-1.
\]
It therefore corresponds to the required matching in $\cH(G)$.
\end{proof}

\section{Deferred proofs from Section~\ref{sec:all-sizes}}
\label{app:profile-proofs}

We prove the three claims deferred from Section~\ref{sec:all-sizes}, using the notation and parameter choices introduced there.

\begin{proof}[Proof of Claim~\ref{prof:sts-extension}]
The case $r=0$ is immediate.  Assume $1\le r\le m/3$, so $n\ge9$, and fix an integer $0\le j<r$.  Let $P$ be a $j$-edge matching on $U=V(P)$, and let $a(P)$ be the number of its one-edge extensions.  For each $v\in U$, let $F_v$ be the family of triples of $S$ containing $v$.  There are $3j$ such families, each of size $(n-1)/2$.  Any two vertices of $U$ lie in exactly one triple, so $\sum_{\{u, v\}\in \binom{U}{2}} |F_u\cap F_v| = \binom{3j}{2}$.  Furthermore, the sum $\sum_{\{u, v, w\}\in \binom{U}{3}} |F_u\cap F_v\cap F_w|$ equals $e_S(U)$, the number of edges contained in $U$. 
The inclusion--exclusion principle thus gives
\begin{equation}\label{prof:sts-extension-formula}
 a(P)=\frac{n(n-1)}6-\frac{3j(n-1)}2+\binom{3j}{2}-e_S(U).
\end{equation}

Define
\begin{equation*}
 b_j:=\frac{(j+1)B_{n,j+1}}{B_{n,j}}
 =\frac{(n-3j)(n-3j-1)(n-3j-2)}{6(n-2)},
 \qquad \varepsilon_S(P):=\frac{a(P)}{b_j}-1.
\end{equation*}
Comparing with \eqref{prof:sts-extension-formula} gives
$a(P)-b_j=\binom{3j}{3}/(n-2)-e_S(U)$.
Since each triple in $U$ uses three distinct pairs, we have
$0\le e_S(U)\le\binom{3j}{2}/3$.  It follows that
\begin{equation}\label{prof:sts-sharp-error}
 -\frac{3j(3j-1)}{(n-3j-1)(n-3j-2)}
 \le\varepsilon_S(P)\le
 \frac{3j(3j-1)(3j-2)}{(n-3j)(n-3j-1)(n-3j-2)}.
\end{equation}

Put $z_j:=3j/(n-3j-2)$.  Since $j\le r-1$ and $r\le m/3\le n/9$, we have
\[
 n-3j-2\ge n-3r+1\ge\frac{2n}{3}+1,
 \quad \text{and} \quad 0\le z_j\le\frac{9j}{2n}\le\frac12.
\]
The bounds in \eqref{prof:sts-sharp-error} therefore imply
$-z_j^2\le\varepsilon_S(P)\le z_j^3$, and hence
$|\varepsilon_S(P)|\le z_j^2\le1/4$.  In particular $a(P)>0$.
Starting with the empty matching, it follows inductively that \(N_j(S)>0\) for every \(0\le j\le r\).

Let $\mathcal M_j$ be the family of $j$-edge matchings in $S$, and define
$\eps'_j:= \frac{1}{|\mathcal{M}_j|}\sum_{P\in\mathcal M_j}\varepsilon_S(P)$.
Averaging preserves the bound $|\eps'_j|\le z_j^2$.  Double counting pairs $(P,e)$ with $P\in\mathcal M_j$ and $e$ extending $P$ gives
\[
 (j+1)N_{j+1}(S)=\sum_{P\in\mathcal M_j}a(P)
 =b_jN_j(S)(1+\eps'_j),
\]
because each $(j+1)$-edge matching contributes one pair for each of its $j+1$ edges.  Dividing by $(j+1)B_{n,j+1}=b_jB_{n,j}$ yields
\begin{equation*}
 \frac{N_{j+1}(S)}{B_{n,j+1}}
 =\frac{N_j(S)}{B_{n,j}}(1+\eps'_j).
\end{equation*}
Iterating from $N_0(S)=B_{n,0}=1$, we obtain
\[
 \frac{N_r(S)}{B_{n,r}}=\prod_{j=0}^{r-1}(1+\eps'_j).
\]
Since $|\eps'_j|\le z_j^2\le1/4$, applying
$\e^{-2|x|}\le1+x\le\e^{2|x|}$ for $|x|\le1/2$ to each factor gives
\[
 \frac{N_r(S)}{B_{n,r}}
 =\exp\left(\pm2\sum_{j=0}^{r-1}z_j^2\right).
\]
The claim follows because
\[
 2\sum_{j=0}^{r-1}z_j^2
 =2\sum_{j=0}^{r-1}\left(\frac{3j}{n-3j-2}\right)^2
 \le\frac{18}{(2n/3)^2}\sum_{j=0}^{r-1}j^2
 \le\frac{27r^3}{2n^2}.\qedhere
\]
\end{proof}

\begin{proof}[Proof of Claim~\ref{prof:restrictions}]
We apply Lemma~\ref{lem:subset-concentration}.
For a fixed row of $L$, condition on the $M$ chosen columns.  Their entries are $M$ distinct symbols.  The retained link count is their intersection with a uniformly random $M$-set of symbols, so its expectation is $M^2/n$ and one exchange changes it by at most one.  Equation~\eqref{prof:subset-tail} bounds its failure probability by $2\exp(-n^{4/3}/(8M))$.  The same argument applies to columns and symbols.  A union bound over all $3n$ possible vertices proves the assertion for $L$.

For $S$, fix $v$ and condition on its selection.  Its link is a perfect matching on the remaining $n-1$ vertices, and the other selected vertices form a uniform $(M-1)$-subset.  Each link pair is fully selected with probability $(M-1)(M-2)/((n-1)(n-2))$, giving mean $\mu_S(M)$.  One exchange destroys at most one complete pair and creates at most one, so $c=2$ is a valid bound.  The conditional failure probability is at most $2\exp(-n^{4/3}/(32(M-1)))$.  A union bound over $v$ proves the assertion for $S$. 

Both total failure probabilities are \(O(n\exp(-c n^{1/3}))\) for a constant \(c>0\), uniformly over \(M\) in the stated range and over the underlying structure.
\end{proof}

\begin{proof}[Proof of Claim~\ref{prof:restriction-tuning}]
Extend the formulas to real $x\in[\eta n/2,n]$.  For $g_L(x):=x(1-d_L(x)^{-a})$, the inequality $xd_L'(x)\le2d_L(x)$ gives
\[
 0<g_L'(x)
 =1-d_L(x)^{-a}+axd_L(x)^{-a-1}d_L'(x)
 \le1-(1-2a)d_L(x)^{-a}<1.
\]
Thus, the integer sequence $\phi_L(M)=\lfloor g_L(M)\rfloor$ is nondecreasing and increases by at most one.  Since $\phi_L(r)<r\le\phi_L(n)$, it attains $r$ for some $M\in[r,n]$.  The floor relation gives $M-r=Md_L(M)^{-a}+O(1)=O_\eta(n^{1-a})$.

For $g_S(x):=x(1-d_S(x)^{-a})/3$, we have $xd_S'(x)\le3d_S(x)$ for sufficiently large $n$, since
\[
 3d_S(x)-xd_S'(x)
 =\frac{x^2-6x+6}{2(n-2)}+6n^{2/3}\ge0.
\]
Consequently $0<g_S'(x)\le[1-(1-3a)d_S(x)^{-a}]/3<1$.  
Thus, the integer sequence $\phi_S(M)=\lfloor g_S(M)\rfloor$ is nondecreasing and increases by at most one.
Since $\phi_S(3r)<r\le\phi_S(n)$, we have $r=\phi_S(M)$ for some  $M\in[3r,n]$,
and $M-3r=Md_S(M)^{-a}+O(1)=O_\eta(n^{1-a})$.  
\end{proof}

\end{document}